\documentclass[11pt]{amsart}
\usepackage{color}
\usepackage{amsmath}
\usepackage[colorlinks=true]{hyperref}

\usepackage{amscd,amsthm,amsfonts,amssymb,esint}
\usepackage{fullpage}
\usepackage[all]{xy}
\usepackage{mathtools}
\usepackage{tikz-cd}

\newcommand\mO{\mathcal{O}}

\newcommand{\mb}[1]{\mathbb{#1}}
\newcommand{\mr}[1]{\mathrm{#1}}
\newcommand{\mc}[1]{\mathcal{#1}}
\newcommand{\ov}{\overline}
\newcommand\ra{\rightarrow}
\newcommand{\ora}[1]{\stackrel{#1}\longrightarrow}

\newcommand\wt{\widetilde}

\newcommand\Spe{\mathrm{Spec\,}}
\newcommand\di{\mathrm{div}}

\begin{document}
	\newtheorem{thm}{Theorem}[section]
	\newtheorem{prop}[thm]{Proposition}
	\newtheorem{lem}[thm]{Lemma}
	\newtheorem{cor}[thm]{Corollary}
	\newtheorem{conj}[thm]{Conjecture}
	\newtheorem{conv}[thm]{Convention}
	\newtheorem{ques}[thm]{Question} 
	\newtheorem{prb}[thm]{Problem}  
	
	\theoremstyle{definition}
	\newtheorem{definition}[thm]{Definition}
	\newtheorem{example}[thm]{Example}
	
	\theoremstyle{remark}
	\newtheorem{remark}[thm]{Remark}
	
	\numberwithin{equation}{section}
	
	\bibliographystyle{alpha}
	
	\title{A high-codimensional Yuan's inequality and its application to higher arithmetic degrees}
	\author{Jiarui Song}
	\address{School of Mathematical Sciences, Peking University, Beijing, P.R.China.}
	\email{soyo999@pku.edu.cn}

	\begin{abstract}
		In this article, we consider a dominant rational self-map $f: X\dashrightarrow X$ of a normal projective variety defined over a number field. We study the arithmetic degree $\alpha_k(f)$ for $f$ and $\alpha_k(f,V)$ of a subvariety $V$, which generalize the classical arithmetic degree $\alpha_1(f,P)$ of a point $P$.
		
		We generalize Yuan's arithmetic version of Siu's inequality to higher codimensions and utilize it to demonstrate the existence of the arithmetic degree $\alpha_k(f)$. Furthermore, we establish the relative degree formula $\alpha_k(f)=\max\{\lambda_k(f),\lambda_{k-1}(f)\}$.
		
		In addition, we prove several basic properties of the arithmetic degree $\alpha_k(f,V)$ and establish the upper bound $\overline{\alpha}_{k+1}(f,V)\leq \max\{\lambda_{k+1}(f),\lambda_{k}(f)\}$, which generalizes the classical result $\overline{\alpha}_f(P)\leq \lambda_1(f)$.
		
		Finally, we discuss a generalized version of the Kawaguchi-Silverman conjecture that was proposed by Dang-Ghioca-Hu-Lesieutre-Satriono, and we provide a counterexample to this conjecture.
	\end{abstract}
	
	\maketitle
	\setcounter{tocdepth}{1}
	\tableofcontents

\section{Introduction}\label{Intro}
We consider a dominant rational self-map $f: X\dashrightarrow X$ of a $d$-dimensional normal projective variety over $\ov{\mathbb{Q}}$. Let $H$ be an ample line bundle on $X$, and we define the \emph{$k$-th dynamical degree} of $f$ as
$$
\lambda_k(f)=\lim_{n\ra \infty}\big(((f^n)^*H^k\cdot H^{d-k})\big)^{1/n}.
$$

In the case of the base field being the complex field $\mathbb{C}$, the existence of the dynamical degree was established by Dinh and Sibony in \cite{dinh_borne_2005}. For any projective normal variety in arbitrary characteristic, Truong proved the existence in \cite{truong_relative_2020} by de Jong's alteration theorem. In \cite{dang_degrees_2020}, Dang established a high-codimensional version of Siu's inequality and utilized this inequality to prove the existence of the dynamical degree.

Another important invariant of $f$ is the \emph{arithmetic degree}. Let $X_f(\overline{\mathbb{Q}})$ be the set of algebraic points on which the $f$-orbit is well-defined. We set $h_X^+=\max\{h_X,1\}$, where $h_X$ is a fixed Weil height on $X$ associated to an ample line bundle, and define the arithmetic degree of $P\in X_f(\overline{\mathbb{Q}})$ as
\[\alpha_f(P)=\lim_{n\ra\infty}h_X^+(f^n(P))^{1/n}.\]
The existence of the limit above is a part of the Kawaguchi-Silverman conjecture, which is currently known only in certain special cases (see \cite{kawaguchi_dynamical_2016},\cite{kawaguchi_dynamical_2016-1}, \cite{meng_kawaguchi-silverman_2022}, \cite{matsuzawa_kawaguchi-silverman_2020}, \cite{jia_endomorphisms_2021}, \cite{wang_periodic_2024}, \cite{matsuzawa_arithmetic_2022}). Thus, we consider the following upper and lower limits:
\[\overline{\alpha}_f(P)=\limsup_{n\ra\infty}h_X^+(f^n(P))^{1/n}\]
and 
\[\underline{\alpha}_f(P)=\liminf_{n\ra\infty}h_X^+(f^n(P))^{1/n}.\]

It is natural to consider the relationship between the dynamical degree and the arithmetic degree of a rational self-map $f$ on a normal projective variety $X$ over $\ov{\mathbb{Q}}$. In \cite{kawaguchi_dynamical_2016-1}, Kawaguchi and Silverman show that $\ov{\alpha}_f(P)\leq \lambda_1(f)$, which unfortunately contains a mistake. In \cite{matsuzawa_upper_2020}, Matsuzawa followed their approach and proved this inequality. Xie in \cite{xie_remarks_2023} later gave another proof over function fields by the relative degree formula. Thus, if we develop the relative degree formula over number fields by Arakelov geometry, we are able to provide a more conceptual proof of this upper bound.

To study the behavior of subvarieties under a rational self-map, it is a natural idea to generalize the concept of dynamical degrees and arithmetic degrees. In their work \cite{dang_higher_2022}, Dang, Ghioca, Hu, Lesieutre and Satriono introduce the concept of higher arithmetic degrees. In this paper, we will concentrate on the high-dimensional case and prove several results related to higher arithmetic degrees.

\subsection{Definitions and main results}
  For a normal projective variety $X$ over $\ov{\mb{Q}}$ and $f:X\dashrightarrow X$, we can take a number field $K$ such that all coefficients of $(X,f)$ are contained in $K$. Then we can assume $X$ is defined over $K$. Furthermore, by the morphism $X\ra \Spe K\ra \Spe \mb{Q}$, we can consider $X$ as a variety over $\mb{Q}$. Note that we do not require a variety to be geometrically integral. By this reduction, we always assume $(X,f)$ are defined over $\mb{Q}$, and the theorems hold for general $X$ defined over a number field. 

 In order to define higher arithmetic degrees, we need to take an arithmetic model $(\mc{X},\ov{\mc{H}})$ of $(X, H)$, where $\ov{\mc{H}}$ is an ample Hermitian line bundle with generic fiber $H$. The dominant rational self-map $f$ of $X$ is lifted to a dominant rational self-map $F$ of $\mathcal{X}$. Let $\Gamma_F$ be the normalization of the graph of $F$ in $\mc{X}\times_{\Spe \mb{Z}}\mc{X}$, and $\pi_1,\pi_2$ be projections from $\Gamma_F$ onto the first and second component respectively. The following diagram commutes.
  \[\begin{tikzcd}
  	& \Gamma_F \arrow[ld, "\pi_1"'] \arrow[rd, "\pi_2"] &   \\
  	\mc{X} \arrow[rr, "F", dashed] &                                                 & \mc{X}
  \end{tikzcd}.\]
  
   Before giving the precise definitions of the higher arithmetic degrees, we record some foundational results to ensure that the corresponding limits are well-defined.
   
   We shall see in Corollary \ref{Cor_welldef} that the limit appearing in the definition of $\alpha_k(f)$ (cf. Definition \ref{Def1}) exists and is independent of the choice of the model $\mc{X}$ and the Hermitian line bundle $\ov{\mc{H}}$. Moreover, $\alpha_k(f)$ is shown to be a birational invariant (see Corollary \ref{birinv}).
   
   In Proposition \ref{welldefined}, we further establish that the upper and lower arithmetic degrees $\ov{\alpha}_{k+1}(f,V)$ and $\underline{\alpha}_{k+1}(f,V)$, whose definitions are also given in Definition~\ref{Def1}, are likewise independent of the choices of $\mc{X}$ and $\ov{\mc{H}}$. However, the existence of the corresponding limit defining $\alpha_{k+1}(f,V)$, as well as the geometric degree $\lambda_k(f,V)$, remains conjectural at present.
   
\begin{definition}\label{Def1}
	With all the notations and constructions above, for any integer $0\leq k\leq d+1$, we define the \emph{$k$-th arithmetic degree} of $F$ with respect to $\ov{\mc{H}}$ as
	\[\widehat{\deg}_k(F)=\big(\pi_2^*\ov{\mc{H}}^{k}\cdot \pi_1^*\ov{\mc{H}}^{d+1-k}\big).\]
	Then the \emph{$k$-th arithmetic degree} of $f$ is defined as
	\[\alpha_k(f)=\lim_{n\ra \infty}\widehat{\deg}_k(F^n)^{1/n}.\]
	Let $V$ be a $k$-dimenisonal subvariety of $X$. Assume the $f$-orbit of $V$ is well-defined. In other words, $f^n(V)$ is not contained in $I(f)$, where $I(f)$ is the indeterminacy locus of $f$. Let $\mc{V}$ denote the Zariski closure of $V$ in $\mc{X}$, we define
	\[\deg_k(f,V)=\big(\pi_2^*H^k\cdot \pi_1^{\#}V\big)\]
	\[\widehat{\deg}_{k+1}(F,\mc{V})=\big(\pi_2^*\ov{\mc{H}}^{k+1}\cdot \pi_1^{\#}\mc{V}\big).\]
	Here the pull-back of a subvariety by a birational map means the strict transform.
	
	The \emph{geometric degree} of $(f,V)$ is defined as 
\[\ov{\lambda}_k(f,V)=\limsup_{n\ra \infty}\deg_k(f^n,V)^{1/n},\]
\[\underline{\lambda}_k(f,V)=\liminf_{n\ra \infty}\deg_k(f^n,V)^{1/n}.\]
If $\ov{\lambda}_k(f,V)=\underline{\lambda}_k(f,V)$, we define 
\[\lambda_k(f,V)=\ov{\lambda}_k(f,V)=\lim_{n\ra \infty}\deg_k(f^n,V)^{1/n}.\]

The \emph{arithmetic degree} of $(f,V)$ is defined as 
\[\ov{\alpha}_{k+1}(f,V)=\limsup_{n\ra\infty}\widehat{\deg}_{k+1}(F^n,\mc{V})^{1/n},\]
\[\underline{\alpha}_{k+1}(f,V)=\liminf_{n\ra\infty}\widehat{\deg}_{k+1}(F^n,\mc{V})^{1/n}.\]
If $\ov{\alpha}_{k+1}(f,V)=\underline{\alpha}_{k+1}(f,V)$, we define 
\[\alpha_{k+1}(f,V)=\ov{\alpha}_{k+1}(f,V)=\lim_{n\ra\infty}\widehat{\deg}_{k+1}(F^n,\mc{V})^{1/n}.\]
\end{definition}

Siu's bigness theorem, also known as Siu's inequality, is a fundamental result in algebraic geometry. It was first proved by Siu in \cite{siu_effective_1993}, and later, Yuan proved an arithmetic version of the bigness theorem in \cite{yuan_big_2008}. Dang developed Theorem \ref{Dangresult} in \cite[Theorem 3.4.4]{dang_degrees_2020}, which can be considered as a high-codimensional generalization of Siu's inequality. Using this inequality, he was able to establish the existence of dynamical degrees. In this paper, we aim to prove an arithmetic analogue of \cite[Theorem 3.4.4]{dang_degrees_2020} as follows. It is also a high codimentional generalization of \cite{yuan_big_2008}.
\begin{thm}[Theorem \ref{Siu2}]\label{Siu2copy}
	Let $\mathcal{X}\ra \Spe\mb{Z}$ be an arithmetic variety with relative dimension $d$, and $r$ and $k$ be positive integers with $r+k\leq d$. Let $\ov{\mc{L}}_1,\ldots, \ov{\mc{L}}_{k}, \ov{\mc{M}},\ov{\mc{N}}_1,\ldots,\ov{\mc{N}_r}$ be nef  Hermitian line bundles. Suppose $\big(\ov{\mc{M}}^{d+1-r}\cdot\ov{\mc{N}}_1\cdots \ov{\mc{N}}_r\big)>0$. Then
	\[	\ov{\mc{L}}_1\cdots \ov{\mc{L}}_k\cdot\ov{\mc{N}}_1\cdots \ov{\mc{N}}_r\leq (d+2-r-k)^k\frac{\big(\ov{\mc{L}}_1,\cdots,\ov{\mc{L}}_k\cdot\ov{\mc{M}}^{d+1-r-k}\cdot\ov{\mc{N}}_1\cdots \ov{\mc{N}}_r\big)}{\big(\ov{\mc{M}}^{d+1-r}\cdot\ov{\mc{N}}_1\cdots \ov{\mc{N}}_r\big)}\cdot \ov{\mc{M}}^{k}\cdot\ov{\mc{N}}_1\cdots \ov{\mc{N}}_r.\]
Here the symbol ``$\leq$" will be explained in Lemma \ref{Siu1}.
\end{thm}
By this result, we are able to establish the submultiplicity formula and prove the existence of $\alpha_k(f)$.
\begin{thm}[Theorem \ref{Submulti2}]\label{Submulti}
	There is a positive constant $C>0$ such that for any dominant rational self-maps $F:\mc{X}\dashrightarrow \mc{X}$ and $G:\mc{X}\dashrightarrow \mc{X}$,
	\[\widehat{\deg}_k(F\circ G)\leq C\cdot\widehat{\deg}_k(F)\cdot\widehat{\deg}_k(G).\]
	The constant $C$ only depends on the integers $k, d=\dim \mc{X}-1$ and the ample Hermitian line bundle $\ov{\mc{H}}$.
\end{thm}
Then the convergence of the limit in $\alpha_k(f)$ follows directly from this theorem and Fekete's lemma \cite[II]{fekete_uber_1923}.

In the function field case, the relative degree formula has been extensively studied, as seen in \cite{dinh_comparison_2011} and \cite{truong_relative_2020}. In \cite[Conjecture 1.8]{dang_higher_2022}, the authors proposed the conjecture that $\alpha_k(f)=\max\{\lambda_k(f),\lambda_{k-1}(f)\}$, which can be viewed as an arithmetic analogue of the relative degree formula, and they proved this conjecture for the case when $\dim X=2$ and $f$ is birational.  In this paper, we will prove the conjecture for general $X$ and $f$.
\begin{thm}[Theorem \ref{RelaDeg2}]\label{RelaDeg}
	Let $f$ be a dominant rational self-map of X. Then for any integer $1\leq k\leq d$,
	one has
\[\alpha_k(f)=\max\{\lambda_k(f),\lambda_{k-1}(f)\}.\]
\end{thm}

The proof of our relative degree theorem relies on a technique involving twisted metrics, which is encapsulated in the following formula:
\[		\big(\ov{\mc{L}}^k\cdot\ov{\mc{M}}(t)^{d+1-k}\big)=\big(\ov{\mc{L}}^k\cdot\ov{\mc{M}}^{d+1-k}\big)+(d+1-k)t\big(L^k\cdot M^{d-k}\big).\]
This formula enables us to establish connections between arithmetic degrees and algebraic degrees. Using this approach, we provide a more uniform proof for both the lower and upper bounds of $\alpha_k(f)$. In particular, our proof of the lower bound $\alpha_k(f)\geq \max\{\lambda_k(f),\lambda_{k-1}(f)\}$ is more straightforward than the argument presented in \cite[Theorem B]{dang_higher_2022}. It is worth noting that Theorem \ref{RelaDeg} serves as an arithmetic analogue of the classical relative degree formula \cite{dinh_comparison_2011}, but only in the specific case where the base scheme is a curve, and the rational self-map on the base is the identity. Furthermore, our strategy differs from the one in algebraic case.

For higher arithmetic degrees, one may ask the following question: Is there an upper bound for $\alpha_{k+1}(f,V)$ in terms of dynamical degrees? The naive analogue $\alpha_{k+1}(f,V)\leq \lambda_{k+1}(f)$ does not hold, as shown in \cite[Example 1.1]{dang_higher_2022}. The following result gives a suitable generalization. 
\begin{thm}[Theorem \ref{UpBdd2}]\label{UpBdd}
	Let $f: X\dashrightarrow X$ be a dominant rational map, and let $V$ be an irreducible subvariety of dimension $k$. Then
	\[\overline{\alpha}_{k+1}(f,V)\leq \alpha_{k+1}(f).\]
\end{thm}
Combining Theorem \ref{RelaDeg} and Theorem \ref{UpBdd}, we obtain an upper bound.
	
When $k=0$, $V$ is a point $P$. We have $\alpha_1(f)=\lambda_1(f)$ and $\ov{\alpha}_1(f,P)=\alpha_f(P)$, which is the classical arithmetic degree. In this case, Theorem \ref{UpBdd} is the fundamental inequality proved in \cite{kawaguchi_dynamical_2016-1} and \cite{matsuzawa_upper_2020}. In other words, we provide a new proof of the fundamental inequality, which is more conceptual and straightforward.

In \cite{dang_higher_2022}, the authors proposed a higher dimensional analogue of Kawaguchi-Silverman conjecture, as stated in Conjecture \ref{Kawa-Sil}. However, this conjecture is false, and we will give a counterexample at the end of this paper. 

\subsection{Organization of the paper}
In Section \ref{SecSiu}, we establish a high-codimensional version of Yuan's inequality, as stated in Theorem \ref{Siu2}. The inequality is applied to prove Theorem \ref{Submulti2}, which guarantees the existence of $\alpha_k(f)$.

In Section \ref{SecRel}, we make use of Theorem \ref{Siu3} to prove the relative degree formula, as stated in Theorem \ref{RelaDeg2}.

In Section \ref{SecUpBd}, we discuss some fundamental properties of the arithmetic degrees of subvarieties. At the end of this section, we prove Theorem \ref{UpBdd2}.

In Section \ref{SecFur}, we discuss the generalized version of the Kawaguchi-Silverman conjecture proposed in \cite[Conjecture 1.6]{dang_higher_2022} and provide a counterexample to this conjecture.

\subsection{Notation and Terminology}\label{Notation}

\textbf{1. }In this paper, a \emph{variety} means an integral separated scheme of finite type over a field. A \emph{subvariety} of a variety $X$ is a closed subset of $X$.  

It should be noted that we do not require a variety to be geometrically integral. Thus a variety over a number field $K$ can be considered as a variety over $\mb{Q}$ via the morphism
\[X\ra \Spe K\ra \Spe\mb{Q}.\]
Therefore, we assume that the variety $X$ is defined over $\mb{Q}$ in this paper. Specifically, $X$ will always be a normal projective variety of dimension $d$ defined over $\mb{Q}$. 

\textbf{2. }Let $f:X\dashrightarrow X$ be a dominant rational map. The set $I(f)$ denotes the indeterminacy locus of $f$. For an irreducible subvariety $V$, we define $f(V)=\overline{f(V\backslash I(f))}$, which is the strict transform. The $f$-orbit of $V$ is well-defined if $f^n(V)$ is not contained in $I(f)$ for any $n>0$.

\textbf{3. }For the arithmetic intersection theory, we follow definitions in \cite{gillet_arithmetic_1990}. 

 By an \emph{arithmetic variety} $\pi: \mc{X}\ra \Spe\mb{Z}$, we mean an integral scheme, flat, separated and of finite type over $\Spe \mb{Z}$. A \emph{Hermitian line bundle} on $\mc{X}$ is the datum $\ov{\mc{L}}=(\mc{L},\|\cdot\|)$ where
 \begin{itemize}
 	\item   $\mathcal{L}$ is a line bundle on $\mathcal{X}$;
 	
 \item  $\|\cdot\|$ is a smooth Hermitian metric on the line bundle $\mathcal{L}(\mathbb{C})$ on $\mathcal{X}(\mathbb{C})$ that is smooth and invariant under complex conjugation.
 \end{itemize}
For a Hermitian line bundle $\ov{\mc{L}}=(\mc{L},\|\cdot\|)$ on $\mc{X}$, we use $L$ to denote its restriction to $X=\mc{X}_{\mb{Q}}$, i.e. $L=\mc{L}_{\mb{Q}}$. For a real number $t$, we denote the Hermitian line bundle $\ov{\mc{L}}(t)=(\mc{L},\|\cdot\|e^{-t})$.

\textbf{4.} For a variety $X$ over a number field $K$ and a line bundle $L$ on $X$, we say that $(\mc{X,\ov{\mc{L}}})$ is an arithmetic model of $(X,L)$ if $\mc{X}$ is an integral scheme, $\mc{X}\ra \Spe O_K$ is a flat projective morphism with generic fiber $\mc{X}\times_{O_K}K\cong X$, and $\ov{\mc{L}}=(\mc{L},\|\cdot\|)$ is a Hermitian line bundle on $\mc{X}$ extending $L$ under the identification $\mc{X}_K\cong X$. We also say $\mc{X}$ is an \emph{integral model} of $X$. Note that the integral model $\mc{X}$ is an arithmetic variety because of the map $\mc{X}\ra \Spe O_K\ra \Spe \mb{Z}$.

\textbf{5.} In this article, we will make use of various positivity conditions on Hermitian line bundles, such as ``ample," ``nef," and ``big." For a comprehensive discussion on the positivity theory of Hermitian line bundles, we refer the reader to \cite{zhang_positive_1995} and \cite{moriwaki_arakelov_2014}. The \emph{small} sections of a Hermitian line bundle $\ov{\mc{L}}$ is the set
\[\widehat{H}^0(\mc{X},\ov{\mc{L}})=\{s\in H^0(\mc{X},\mc{L}): \|s\|_{\mr{sup}}\leq 1\},\]
where $\|s\|_{\mr{sup}}=\sup\limits_{z\in \mc{X}(\mb{C})}\|s(z)\|$. A Hermitian line bundle is called \emph{effective} if $\widehat{H}^0(\mc{X},\ov{\mc{L}})$ is nonzero.

\textbf{6.} On an arithmetic variety $\mc{X}\ra\Spe\mb{Z}$, for $\ov{\mc{Z}}\in \widehat{\mr{CH}}^k(\mc{X})$, we say $\ov{\mc{Z}}\geq 0$ if for any nef Hermitian line bundles $\ov{\mc{H}}_1,\ldots,\ov{\mc{H}}_{d+1-k}$, the arithmetic intersection number
\[\big(\ov{\mc{H}}_1\cdots\ov{\mc{H}}_{d+1-k}\cdot \ov{\mc{Z}}\big)\geq 0.\]

\subsection{Acknowledgements}
I will express great gratitude to Junyi Xie for encouraging me to work on this problem and giving me lots of helpful advice. I am grateful to Xinyi Yuan who suggested using the arithmetic Demailly approximation to tackle this problem.  Additionally, I would like to acknowledge Chunhui Liu for meticulously examining the earlier version of this article and Shu Kawaguchi, John Lesieutre, Yohsuke Matsuzawa, Joseph Silverman for providing valuable suggestions. I would like to thank Shou-wu Zhang for answering my questions and my classmate Ruoyi Guo for some useful discussions about Arakelov geometry.

\section{The high-codimensional Yuan's inequality}\label{SecSiu}	
In \cite{dang_degrees_2020}, Dang proves the following result:
\begin{thm}[\cite{dang_degrees_2020}, Theorem 3.4.4]\label{Dangresult}
	Let $X$ be a normal projective variety of dimension $d$ over an algebraically closed field and $1\leq i\leq d$ be an integer. Then for any nef line bundles $\alpha_1, \ldots, \alpha_i$ and  big and nef line bundle $\beta$ on $X$, one has  
	$$
	\alpha_1\cdot\alpha_2\cdots \alpha_i \leqslant(d-i+1)^i \frac{\left(\alpha_1\cdots \alpha_i \cdot \beta^{d-i}\right)}{\left(\beta^d\right)} \times \beta^i .
	$$
	Here we write $\alpha\leq \beta$ for $\alpha,\beta\in N^i(X)$ if $\beta-\alpha$ is pseudo effective.
\end{thm}
Here we follow the intersection theory in \cite[Section 2.3]{fulton_intersection_1998}.

In the proof of Theorem \ref{Dangresult}, we need to cut the divisors by certain sections of other line bundles. However, a section of a Hermitian line bundle is not only a topological space, but also contains information of the metric.  To overcome this problem, we make use of the following arithmetic Demailly approximation theorem, which allows us to omit the metric part without losing essential properties.
\begin{thm}[\cite{qu_arithmetic_2024}, Theorem 1.5]\label{ArithDe}
	Let $\mathcal{X}\ra\Spe \mb{Z}$ be an arithmetic variety, normal and generically smooth. Let $\overline{\mathcal{L}}$ be an ample Hermitian line bundle. Let $Y$ be a closed subset of $\mathcal{X}_\mathbb{Q}$.
	Then there exists $s_{n_i} \in H^0(\mathcal{X},n_i\mathcal{L})$ such that
	\begin{itemize}
		\item $\|s_{n_i}\|_\infty \rightarrow 1$ and $\frac1{n_i} \log\|s_{n_i}\| \rightarrow 0$ in $L^1$-topology.
		\item $\operatorname{div}(s_{n_i})_\mathbb{Q}$ is smooth and does not contain any irreducible components of $Y$.
		\item $\operatorname{div}(s_{n_i})$ has no vertical components.
	\end{itemize}
\end{thm}

\hspace*{\fill} \\

The following is our key lemma.

\begin{lem}\label{Siu1}
	Let $\mathcal{X}\ra \Spe\mb{Z}$ be an arithmetic variety with relative dimension $d$, and $0\leq r\leq d-1$ be an integer. Let $\ov{\mc{L}},\ov{\mc{M}}$ and $\ov{\mc{N}}_i$, $i=1,2,\ldots,r$ be nef Hermitian line bundles such that 
	
	\[\big(\ov{\mc{M}}^{d+1-r}\cdot \ov{\mc{N}}_1\cdots \ov{\mc{N}}_r\big)>(d+1-r)\cdot \big(\ov{\mc{M}}^{d-r}\cdot\ov{\mc{L}}\cdot \ov{\mc{N}}_1\cdots \ov{\mc{N}}_r\big).\] Then there is an inequality 
	\[
	\ov{\mc{L}}\cdot \ov{\mc{N}}_1\cdots \ov{\mc{N}}_r\leq \ov{\mc{M}}\cdot \ov{\mc{N}}_1\cdots \ov{\mc{N}}_r.
	\]
	Here for $\ov{\mc{Z}}\in \widehat{\mr{CH}}^k(\mc{X})$, we say $\ov{\mc{Z}}\geq 0$ if for any nef Hermitian line bundles $\ov{\mc{H}}_1,\ldots,\ov{\mc{H}}_{d+1-k}$, the arithmetic intersection number
	\[\big(\ov{\mc{H}}_1\cdots\ov{\mc{H}}_{d+1-k}\cdot \ov{\mc{Z}}\big)\geq 0.\]

For $r=0$, the inequality means $\ov{\mc{L}}\leq \ov{\mc{M}}$.
\end{lem}
\begin{proof}
We prove this inequality by induction on $r$. When $r=0$, we have to prove $\ov{\mc{L}}\leq \ov{\mc{M}}$
	under the assumption
	\[\big(\ov{\mc{M}}^{d+1}\big)>(d+1)\big(\ov{\mc{M}}^d\cdot \ov{\mc{L}}\big).\]
	By \cite[Theorem 2.2]{yuan_big_2008}, $\ov{\mc{M}}-\ov{\mc{L}}$ is a big Hermitian line bundle, which implies
	\[\big(\ov{\mc{H}}_1\cdots\ov{\mc{H}}_{d}\cdot \ov{\mc{M}}\big)\geq \big(\ov{\mc{H}}_1\cdots\ov{\mc{H}}_{d}\cdot \ov{\mc{L}}\big)\]
	for any nef Hermitian line bundles $\ov{\mc{H}}_1,\ldots,\ov{\mc{H}}_{d}$.
	
	Assume $r\geq 1$ and the conclusion holds for $r-1$. Additionally, to prove the inequality, we can assume $\mc{X}$ is normal and generically smooth. Indeed, we take a dominant generically finite morphism $\pi:\mc{X}'\ra \mc{X}$ with $\mc{X}'$ normal and generically smooth. Since $\mc{X}^{\prime}$ integral and dominates $ \Spe\mb{Z}$, the morphism $\mc{X}^{\prime}\ra \Spe\mb{Z}$ is flat by \cite[Proposition 3.9.7]{hartshorne_algebraic_1977}. As a result, $\mc{X}^{\prime}$ is still an arithmetic variety and we can replace $\mc{X}$ by $\mc{X}^{\prime}$.
	
	 Denote the topological degree of $\pi$ by $\delta$. Then by the projection formula (See \cite[Proposition 2.3.1(iv)]{bost_heights_1994}),
	\begin{equation}
	\begin{aligned}
		&\big(\pi^*\ov{\mc{M}}^{d+1-r}\cdot \pi^*\ov{\mc{N}}_1\cdots \pi^*\ov{\mc{N}}_r\big)-(d+1-r)\cdot \big(\pi^*\ov{\mc{M}}^{d-r}\cdot\pi^*\ov{\mc{L}}\cdot \pi^*\ov{\mc{N}}_1\cdots \pi^*\ov{\mc{N}}_r\big)\\
		=& \delta\cdot \big(\big(\ov{\mc{M}}^{d+1-r}\cdot \ov{\mc{N}}_1\cdots \ov{\mc{N}}_r\big)-(d+1-r)\cdot \big(\ov{\mc{M}}^{d-r}\cdot\ov{\mc{L}}\cdot \ov{\mc{N}}_1\cdots \ov{\mc{N}}_r\big)\big)	> 0.
	\end{aligned}
	\end{equation}
	If the result holds on $\mc{X}'$, then
	\[\big(\pi^*\ov{\mc{H}}_1\cdots\pi^*\ov{\mc{H}}_{d-r}\cdot \pi^*\ov{\mc{M}}\cdot \pi^*\ov{\mc{N}}_1\cdots \pi^*\ov{\mc{N}}_r\big)\geq \big(\pi^*\ov{\mc{H}}_1\cdots\pi^*\ov{\mc{H}}_{d-r}\cdot \pi^*\ov{\mc{L}}\cdot \pi^*\ov{\mc{N}}_1\cdots \pi^*\ov{\mc{N}}_r\big),\]
	which implies
	\[\big(\ov{\mc{H}}_1\cdots\ov{\mc{H}}_{d-r}\cdot \ov{\mc{M}}\cdot \ov{\mc{N}}_1\cdots \ov{\mc{N}}_r\big)\geq \big(\ov{\mc{H}}_1\cdots\ov{\mc{H}}_{d-r}\cdot \ov{\mc{L}}\cdot \ov{\mc{N}}_1\cdots \ov{\mc{N}}_r\big).\]
	
	Moreover, since nef Hermitian line bundles can be approximated by ample line bundles, we can assume all these line bundles are ample. (The approximation means that take an ample Hermitian line bundle $\ov{\mc{H}}$, replace the nef Hermitian line bundle $\ov{\mc{L}}$ by $\ov{\mc{L}}+\varepsilon\ov{\mc{H}}$, and take $\varepsilon\ra 0$.)
	
	By Theorem \ref{ArithDe}, there exists a sequence of integral sections $(s_{n_i})_{n_i\in \mb{N}}$, $s_{n_i}\in H^0(\mc{X},n_i\mc{N}_1)$ satisfying
	\begin{itemize}
		\item  $\frac{1}{n_i}\log\|s_{n_i}\|\ra 0$ in $L^1$-topology on $\mc{X}(\mb{C})$,
		\item $\di(s_{n_i})_{\mb{Q}}$ smooth,
		 \item $\di(s_{n_i})$ has no vertical components.
	\end{itemize}
 This implies $\di(s_{n_i})$ is horizontal and irreducible. We denote $\mc{Y}=\di(s_{n_i})$, which is an arithmetic variety of dimension $d+1-r$ with smooth generic fiber.  Now we compute the intersection numbers.
 \begin{equation}
 	\begin{aligned}
 		\big(\ov{\mc{M}}^{d+1-r}\cdot \ov{\mc{N}}_1\cdots \ov{\mc{N}}_r\big)=& \frac{1}{n_i}\big(\ov{\mc{M}}^{d+1-r}\cdot \ov{\mc{N}}_2\cdots\ov{\mc{N}}_r\cdot\mc{Y}\big)\\
 		&-\frac{1}{n_i}\int_{\mc{X}(\mb{C})}\log\|s_{n_i}\|c_1(\ov{\mc{M}})^{d+1-r}c_1(\ov{\mc{N}}_2)\cdots c_1(\ov{\mc{N}}_r)
 	\end{aligned}
 \end{equation}
 and 
 \begin{equation}
 	\begin{aligned}
 		\big(\ov{\mc{M}}^{d-r}\cdot\ov{\mc{L}}\cdot \ov{\mc{N}}_1\cdots \ov{\mc{N}}_r\big)=& \frac{1}{n_i}\big(\ov{\mc{M}}^{d-r}\cdot\ov{\mc{L}}\cdot\ov{\mc{N}}_2\cdots\ov{\mc{N}}_r \cdot\mc{Y}\big)\\
 		&-\frac{1}{n_i}\int_{\mc{X}(\mb{C})}\log\|s_{n_i}\|c_1(\ov{\mc{M}})^{d-r}c_1(\ov{\mc{L}})c_1(\ov{\mc{N}}_2)\cdots c_1(\ov{\mc{N}}_r).
 	\end{aligned}
 \end{equation}
 Since $\frac{\log\|s_{n_i}\|}{n_i}\ra 0$ in $L^1$-topology and these Chern forms are smooth, we have
 \[\lim_{i\ra \infty}\frac{1}{n_i}\int_{\mc{X}(\mb{C})}\log\|s_{n_i}\|c_1(\ov{\mc{M}})^{d+1-r}c_1(\ov{\mc{N}}_2)\cdots c_1(\ov{\mc{N}}_r)=0,\]
  \[\lim_{i\ra \infty}\frac{1}{n_i}\int_{\mc{X}(\mb{C})}\log\|s_{n_i}\|c_1(\ov{\mc{M}})^{d-r} c_1(\ov{\mc{L}})c_1(\ov{\mc{N}}_2)\cdots c_1(\ov{\mc{N}}_r)=0.\]
 Thus
 \begin{equation}
 	\begin{aligned}
 		&\big(\ov{\mc{M}}^{d+1-r}\cdot \ov{\mc{N}}_1\cdots \ov{\mc{N}}_r\big)-(d+1-r)\cdot \big(\ov{\mc{M}}^{d-r}\cdot\ov{\mc{L}}\cdot \ov{\mc{N}}_1\cdots \ov{\mc{N}}_r\big)\\
 		=&\frac{1}{n_i}\bigg(\big(\ov{\mc{M}}^{d+1-r}\cdot\ov{\mc{N}}_2\cdots\ov{\mc{N}}_r \cdot\mc{Y} \big)-(d+1-r)\big(\ov{\mc{M}}^{d-r}\cdot \ov{\mc{L}}\cdot \ov{\mc{N}}_2\cdots\ov{\mc{N}}_r\cdot \mc{Y}\big)\bigg)+o(1)
 	\end{aligned}
 \end{equation}
 as $n_i \ra \infty$. Taking $n_i$ sufficiently large, we can make 
 \[\big(\ov{\mc{M}}^{d+1-r}\cdot\ov{\mc{N}}_2\cdots\ov{\mc{N}}_r \cdot\mc{Y} \big)>(d+1-r)\big(\ov{\mc{M}}^{d-r}\cdot \ov{\mc{L}}\cdot \ov{\mc{N}}_2\cdots\ov{\mc{N}}_r\cdot \mc{Y}\big).\]
 Then we apply the induction hypothesis for $r-1$ to the arithmetic variety $\mc{Y}=\di(s_{n_i})$ and Hermitian line bundles $\ov{\mc{M}}|_{\mc{Y}}, \ov{\mc{L}}|_{\mc{Y}},\ov{\mc{N}}_2|_{\mc{Y}},\ldots, \ov{\mc{N}}_r|_{\mc{Y}}$. For fixed nef Hermitian line bundles $\ov{\mc{H}}_1,\ldots,\ov{\mc{H}}_{d-r}$, we have
\[\big(\ov{\mc{H}}_1\cdots\ov{\mc{H}}_{d-r}\cdot \ov{\mc{M}}\cdot \ov{\mc{N}}_2\cdots \ov{\mc{N}}_r\cdot \mc{Y}\big)\geq \big(\ov{\mc{H}}_1\cdots\ov{\mc{H}}_{d-r}\cdot \ov{\mc{L}}\cdot \ov{\mc{N}}_2\cdots \ov{\mc{N}}_r\cdot \mc{Y}\big).\]

By a similar computation,
\begin{equation}
	\begin{aligned}
		&\big(\ov{\mc{H}}_1\cdots\ov{\mc{H}}_{d-r}\cdot \ov{\mc{M}}\cdot \ov{\mc{N}}_1\cdots \ov{\mc{N}}_r\big)-\big(\ov{\mc{H}}_1\cdots\ov{\mc{H}}_{d-r}\cdot \ov{\mc{L}}\cdot \ov{\mc{N}}_1\cdots \ov{\mc{N}}_r\big)\\
		=&\frac{1}{n_i}\bigg(\big(\ov{\mc{H}}_1\cdots\ov{\mc{H}}_{d-r}\cdot \ov{\mc{M}}\cdot \ov{\mc{N}}_2\cdots \ov{\mc{N}}_r\cdot\mc{Y}\big)-\big(\ov{\mc{H}}_1\cdots\ov{\mc{H}}_{d-r}\cdot \ov{\mc{L}}\cdot\ov{\mc{N}}_2\cdots \ov{\mc{N}}_r\cdot\mc{Y}\big)\bigg)+o(1)
	\end{aligned}
\end{equation}
when $n_i \ra \infty$. Taking $n_i\ra \infty$, we get 
\[\big(\ov{\mc{H}}_1\cdots\ov{\mc{H}}_{d-r}\cdot \ov{\mc{M}}\cdot \ov{\mc{N}}_1\cdots \ov{\mc{N}}_r\big)\geq\big(\ov{\mc{H}}_1\cdots\ov{\mc{H}}_{d-r}\cdot \ov{\mc{L}}\cdot \ov{\mc{N}}_1\cdots \ov{\mc{N}}_r\big),\]
which completes the induction.
\end{proof}

\begin{cor}\label{Siu1cor}
		Let $\mathcal{X}\ra \Spe\mb{Z}$ be an arithmetic variety with relative dimension $d$, and $0\leq r\leq d-1$ be an integer. Let $\ov{\mc{L}},\ov{\mc{M}}$ and $\ov{\mc{N}}_i$, $i=1,2,\ldots,r$ be nef Hermitian line bundles such that the arithmetic intersection $\big(\ov{\mc{M}}^{d+1-r}\cdot \ov{\mc{N}}_1\cdots \ov{\mc{N}}_r\big)>0$, then 
	\[
	\ov{\mc{L}}\cdot \ov{\mc{N}}_1\cdots \ov{\mc{N}}_r\leq (d+1-r) \frac{\big(\ov{\mc{M}}^{d-r}\cdot\ov{\mc{L}}\cdot \ov{\mc{N}}_1\cdots \ov{\mc{N}}_r\big)}{\big(\ov{\mc{M}}^{d+1-r}\cdot \ov{\mc{N}}_1\cdots \ov{\mc{N}}_r\big)}\cdot \ov{\mc{M}}\cdot \ov{\mc{N}}_1\cdots \ov{\mc{N}}_r.
	\]
\end{cor}
\begin{proof}
	 We take $\lambda>(d+1-r)\frac{\big(\ov{\mc{M}}^{d-r}\cdot\ov{\mc{L}}\cdot \ov{\mc{N}}_1\cdots \ov{\mc{N}}_r\big)}{\big(\ov{\mc{M}}^{d+1-r}\cdot \ov{\mc{N}}_1\cdots \ov{\mc{N}}_r\big)}$, then
	\[\big((\lambda\ov{\mc{M}})^{d+1-r}\cdot \ov{\mc{N}}_1\cdots \ov{\mc{N}}_r\big)>(d+1-r)\cdot \big((\lambda\ov{\mc{M}})^{d-r}\cdot\ov{\mc{L}}\cdot \ov{\mc{N}}_1\cdots \ov{\mc{N}}_r\big).\]
	By Lemma \ref{Siu1}, there is an inequality
	\[ \lambda\ov{\mc{M}}\cdot \ov{\mc{N}}_1\cdots \ov{\mc{N}}_r\geq  \ov{\mc{L}}\cdot \ov{\mc{N}}_1\cdots \ov{\mc{N}}_r.\]
	Let $\lambda\ra (d+1-r)\frac{\big(\ov{\mc{M}}^{d-r}\cdot\ov{\mc{L}}\cdot \ov{\mc{N}}_1\cdots \ov{\mc{N}}_r\big)}{\big(\ov{\mc{M}}^{d+1-r}\cdot \ov{\mc{N}}_1\cdots \ov{\mc{N}}_r\big)}$, one gets
	\[
\ov{\mc{L}}\cdot \ov{\mc{N}}_1\cdots \ov{\mc{N}}_r\leq (d+1-r) \frac{\big(\ov{\mc{M}}^{d-r}\cdot\ov{\mc{L}}\cdot \ov{\mc{N}}_1\cdots \ov{\mc{N}}_r\big)}{\big(\ov{\mc{M}}^{d+1-r}\cdot \ov{\mc{N}}_1\cdots \ov{\mc{N}}_r\big)}\cdot \ov{\mc{M}}\cdot \ov{\mc{N}}_1\cdots \ov{\mc{N}}_r.
\]
\end{proof}

\begin{thm}\label{Siu2}
	Let $\mathcal{X}\ra \Spe\mb{Z}$ be an arithmetic variety with relative dimension $d$, and $r$ and $k$ be positive integers with $r+k\leq d$. Let $\ov{\mc{L}}_1,\ldots, \ov{\mc{L}}_{k}, \ov{\mc{M}},\ov{\mc{N}}_1,\ldots,\ov{\mc{N}_r}$ be nef  Hermitian line bundles. Suppose $\big(\ov{\mc{M}}^{d+1-r}\cdot\ov{\mc{N}}_1\cdots \ov{\mc{N}}_r\big)>0$. Then
\[	\ov{\mc{L}}_1\cdots \ov{\mc{L}}_k\cdot\ov{\mc{N}}_1\cdots \ov{\mc{N}}_r\leq (d+2-r-k)^k\frac{\big(\ov{\mc{L}}_1\cdots\ov{\mc{L}}_k\cdot\ov{\mc{M}}^{d+1-r-k}\cdot\ov{\mc{N}}_1\cdots \ov{\mc{N}}_r\big)}{\big(\ov{\mc{M}}^{d+1-r}\cdot\ov{\mc{N}}_1\cdots \ov{\mc{N}}_r\big)}\cdot \ov{\mc{M}}^{k}\cdot\ov{\mc{N}}_1\cdots \ov{\mc{N}}_r.\]
	
\end{thm}

\begin{proof}
	Since nef Hermitian line bundles can be approximated by ample Hermitian line bundles, we can assume all line bundles above are ample. Taking $\ov{\mc{N}}_1\cdots\ov{\mc{N}}_r$ to be $\ov{\mc{M}}^{i-1}\cdot \ov{\mc{L}}_{i+1}\cdots \ov{\mc{L}}_k\cdot\ov{\mc{N}}_1\cdots\ov{\mc{N}}_r$ and $\ov{\mc{L}}$ to be $\ov{\mc{L}}_i$ in Corollary \ref{Siu1cor}, we get the following inequality for $1\leq i\leq k$.
	\[\ov{\mc{M}}^{i-1}\cdot\ov{\mc{L}}_{i}\cdots \ov{\mc{L}}_k\cdot\ov{\mc{N}}_1\cdots \ov{\mc{N}}_r\leq
	(d+2-r-k) \frac{\big(\ov{\mc{L}}_i\cdots \ov{\mc{L}}_k\cdot\ov{\mc{M}}^{d+i-r-k}\cdot \ov{\mc{N}}_1\cdots \ov{\mc{N}}_r\big)}{\big(\ov{\mc{L}}_{i+1}\cdots \ov{\mc{L}}_k\cdot\ov{\mc{M}}^{d+i-r-k+1}\cdot \ov{\mc{N}}_1\cdots \ov{\mc{N}}_r\big)}\cdot \ov{\mc{M}}^i\cdot\ov{\mc{L}}_{i+1}\cdots \ov{\mc{L}}_k\cdot\ov{\mc{N}}_1\cdots \ov{\mc{N}}_r.\]
	Composing all $k$ inequalities for $1\leq i\leq k$, we obtain
	\begin{equation}
		\begin{aligned}
			\ov{\mc{L}}_1\cdots \ov{\mc{L}}_k\cdot\ov{\mc{N}}_1\cdots \ov{\mc{N}}_r&\leq(d+2-r-k) \frac{\big(\ov{\mc{L}}_1\cdots \ov{\mc{L}}_k\cdot\ov{\mc{M}}^{d+1-r-k}\cdot \ov{\mc{N}}_1\cdots \ov{\mc{N}}_r\big)}{\big(\ov{\mc{L}}_{2}\cdots \ov{\mc{L}}_k\cdot\ov{\mc{M}}^{d+2-r-k}\cdot \ov{\mc{N}}_1\cdots \ov{\mc{N}}_r\big)}\cdot \ov{\mc{M}}\cdot\ov{\mc{L}}_{2}\cdots \ov{\mc{L}}_k\cdot\ov{\mc{N}}_1\cdots \ov{\mc{N}}_r\\
			&\leq(d+2-r-k)^2 \frac{\big(\ov{\mc{L}}_1\cdots \ov{\mc{L}}_k\cdot\ov{\mc{M}}^{d+1-r-k}\cdot \ov{\mc{N}}_1\cdots \ov{\mc{N}}_r\big)}{\big(\ov{\mc{L}}_{3}\cdots \ov{\mc{L}}_k\cdot\ov{\mc{M}}^{d+3-r-k}\cdot \ov{\mc{N}}_1\cdots \ov{\mc{N}}_r\big)}\cdot \ov{\mc{M}}^2\cdot\ov{\mc{L}}_{3}\cdots \ov{\mc{L}}_k\cdot\ov{\mc{N}}_1\cdots \ov{\mc{N}}_r\\
			&\cdots\\
			&\leq (d+2-r-k)^k\frac{\big(\ov{\mc{L}}_1\cdots\ov{\mc{L}}_k\cdot\ov{\mc{M}}^{d+1-r-k}\cdot\ov{\mc{N}}_1\cdots \ov{\mc{N}}_r\big)}{\big(\ov{\mc{M}}^{d+1-r}\cdot\ov{\mc{N}}_1\cdots \ov{\mc{N}}_r\big)}\cdot \ov{\mc{M}}^{k}\cdot\ov{\mc{N}}_1\cdots \ov{\mc{N}}_r.
		\end{aligned}
	\end{equation}
\end{proof}

In most situations, we consider the case where $r=0$, and we state it as the following corollary.

\begin{cor}\label{Siu3}
	Let $\ov{\mc{L}}_1,\ldots,\ov{\mc{L}}_k,\ov{\mc{M}}$ be nef Hermitian line bundles on $\mc{X}$ and $k$ be a positive integer with $k\leq d$. Suppose $\big(\ov{\mc{M}}^{d+1}\big)>0$, then 
	\[ \ov{\mc{L}}_1\cdots \ov{\mc{L}}_k\leq
	(d+2-k)^k\frac{\big(\ov{\mc{L}}_1,\cdots,\ov{\mc{L}}_k\cdot\ov{\mc{M}}^{d+1-k}\big)}{\big(\ov{\mc{M}}^{d+1}\big)}\cdot \ov{\mc{M}}^{k}.\]
\end{cor}

Now we use Corollary \ref{Siu3} to prove Theorem \ref{Submulti2}.

\begin{thm}\label{Submulti2}
	There is a positive constant $C>0$ such that for any dominant rational self-maps $F:\mc{X}\dashrightarrow \mc{X}$ and $G:\mc{X}\dashrightarrow \mc{X}$,
	\[\widehat{\deg}_k(F\circ G)\leq C\cdot\widehat{\deg}_k(F)\cdot\widehat{\deg}_k(G).\]
	The constant $C$ only depends on the integers $k, d=\dim \mc{X}$ and the big and nef Hermitian line bundle $\ov{\mc{H}}$.
\end{thm}

\begin{proof}
	Let $\Gamma_G$ (resp. $\Gamma_F$) be the normalization of the graph of $G$ (resp. $F$), and $\pi_1,\pi_2$ (resp. $\pi_3, \pi_4$) be projections to the first and second coordinate. Let $\Gamma$ be the normalization of the graph of $\pi_3^{-1}\circ \pi_2$ and $u,v$ are projections to $\Gamma_G,\Gamma_F$ respectively. Then we have the following diagram.
	
	
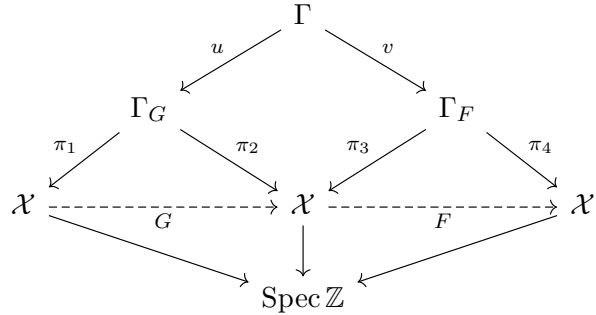
\[\begin{tikzcd}
		&                                                   & \Gamma \arrow[ld, "u"'] \arrow[rd, "v"]   &                                                   &                    \\
		& \Gamma_G \arrow[ld, "\pi_1"'] \arrow[rd, "\pi_2"] &                                           & \Gamma_F \arrow[ld, "\pi_3"'] \arrow[rd, "\pi_4"] &                    \\
		\mc{X} \arrow[rrd] \arrow[rr, "G"', dashed] &                                                   & \mc{X} \arrow[d] \arrow[rr, "F"', dashed] &                                                   & \mc{X} \arrow[lld] \\
		&                                                   & \Spe \mb{Z}                              &                                                   &                   
	\end{tikzcd}\]
	The morphisms $\pi_1,\pi_3,u$ are birational and $\pi_2,\pi_4,v$ are generically finite.
	Applying Corollary \ref{Siu3} to $v^*\pi_4^*\ov{\mc{H}}^k$ and $u^*\pi_2^*\ov{\mc{H}}^k$, we obtain
	\[v^*\pi_4^*\ov{\mc{H}}^k\leq (d+2-k)^k\frac{\big(u^*\pi_2^*\ov{\mc{H}}^{d+1-k}\cdot v^*\pi_4^*\ov{\mc{H}}^k\big)}{\big(u^*\pi_2^*\ov{\mc{H}}^{d+1}\big)}\cdot u^*\pi_2^*\ov{\mc{H}}^k.\]
	Intersecting both sides of the inequality above with $u^*\pi_1^*\ov{\mc{H}}^{d+1-k}$, one gets
	\[\widehat{\deg}_k(F\circ G)\leq (d+2-k)^k\frac{\big(u^*\pi_2^*\ov{\mc{H}}^{d+1-k}\cdot v^*\pi_4^*\ov{\mc{H}}^k\big)}{\big(u^*\pi_2^*\ov{\mc{H}}^{d+1}\big)} \widehat{\deg}_k(G).\]
	Note that $\pi_2\circ u=\pi_3\circ v$. Denote the topological degree of $v$ by $\delta$, and we have 
	\[\frac{\big(u^*\pi_2^*\ov{\mc{H}}^{d+1-k}\cdot v^*\pi_4^*\ov{\mc{H}}^k\big)}{\big(u^*\pi_2^*\ov{\mc{H}}^{d+1}\big)}= \frac{\big(v^*\pi_3^*\ov{\mc{H}}^{d+1-k}\cdot v^*\pi_4^*\ov{\mc{H}}^k\big)}{\big(v^*\pi_3^*\ov{\mc{H}}^{d+1}\big)}=\frac{\delta\cdot\big(\pi_3^*\ov{\mc{H}}^{d+1-k}\cdot \pi_4^*\ov{\mc{H}}^k\big)}{\delta\cdot\big(\pi_3^*\ov{\mc{H}}^{d+1}\big)}=\frac{\widehat{\deg}_k(F)}{\big(\ov{\mc{H}}^{d+1}\big)}.\]
	Thus
	\[\widehat{\deg}_k(F\circ G)\leq C
	\cdot \widehat{\deg}_k(F)\cdot\widehat{\deg}_k(G),\]
	where $C=\frac{(d+1-k)^k}{\big(\ov{\mc{H}}^{d+1}\big)}>0$ is a constant only relies on $d,k$ and $\ov{\mc{H}}$.
\end{proof}

\begin{cor}\label{Cor_welldef}
Let $f: X\dashrightarrow X$ be a dominant rational self-map of $X$. Consider the $k$-th arithmetic degree
	\[\alpha_k(f)=\lim\limits_{n\ra \infty}\widehat{\deg}_k(f^n)^{1/n}.\]
	Then the limit exists and is independent of the choice of the model $\mc{X}$ or the ample Hermitian line bundle $\ov{\mc{H}}$.
\end{cor}
	\begin{proof}
		By applying Fekete's lemma \cite[II]{fekete_uber_1923} to Theorem \ref{Submulti2}, we can establish the existence of the limit
	\[\alpha_k(f)=\lim\limits_{n\ra \infty}\widehat{\deg}_k(f^n)^{1/n}.\]
 Our remaining task is to demonstrate its independence.
	
		Suppose $\mc{X}_1$ and $\mc{X}_2$ are two integral models of $X$, $\ov{\mc{H}}_1$ and $\ov{\mc{H}}_2$ are ample Hermitian line bundles on $\mc{X}_1$ and $\mc{X}_2$.
		We take $\mc{X}_3$ to be the Zariski closure of the image of the composition
		\[X\ora{\Delta}X\times_{\Spe\mb{Q}}X\ra \mc{X}\times_{\Spe \mb{Z}}\mc{X}.\]
		Denote by $p_1:\mc{X}_3\ra \mc{X}_1, p_2:\mc{X}_3\ra \mc{X}_2$ the projections, which are birational. There is a commutative diagram
		\[\begin{tikzcd}
			\mc{X}_1 \arrow[d, "F_1"', dashed] & \mc{X}_3 \arrow[d, "F_3", dashed] \arrow[l, "p_1"'] \arrow[r, "p_2"] & \mc{X}_2 \arrow[d, "F_2", dashed] \\
			\mc{X}_1                           & \mc{X}_3 \arrow[l, "p_1"'] \arrow[r, "p_2"]                          & \mc{X}_2                         
		\end{tikzcd}.\]
		Fix an ample Hermitian line bundle $\ov{\mc{H}}_3$ on $\mc{X}_3$, then
		\[\widehat{\deg}_{k}(F_3^n)=\widehat{\deg}_{k}(p_1^{-1}\circ F_1^n\circ p_1)\leq C\cdot \widehat{\deg}_{k}(p_1^{-1})\cdot \widehat{\deg}_{k}(F_1^n)\cdot \widehat{\deg}_{k}(p_1)\]
		for some constant $C>0$. We also have
		\[\widehat{\deg}_{k}(F_1^n)=\widehat{\deg}_{k}(p_1\circ F_3^n\circ p_1^{-1})\leq C'\cdot \widehat{\deg}_{k}(p_1)\cdot \widehat{\deg}_{k}(F_3^n)\cdot \widehat{\deg}_{k}(p_1^{-1}).\]
		Hence 
		\[\lim_{n\ra \infty}\widehat{\deg}_{k}(F_1^n)^{1/n}=\lim_{n\ra \infty}\widehat{\deg}_{k}(F_3^n)^{1/n}.\]
		In other words, the arithmetic degree defined on $\mc{X}_1$ and $\mc{X}_3$ are the same. Similarly,  the arithmetic degree defined on $\mc{X}_2$ and $\mc{X}_3$ are the same. Thus $\alpha_k(f)$ is independent of the choice integral models or ample Hermitian line bundles.
	\end{proof}

The birational invariance of $\alpha_k(f)$ will be proved in Corollary \ref{birinv} by the relative degree formula.

\section{The relative degree formula}\label{SecRel}
In this section, we present the proof of Theorem \ref{RelaDeg2}. We begin by proving a lemma that is similar to Corollary \ref{Siu3}. Throughout the proof, we will make use of the following fact several times: for $t>0$
\[\big(\ov{\mc{L}}_1\cdots\ov{\mc{L}}_d\cdot \ov{\mO}_{\mc{X}}(t)\big)=t\int_{\mc{X}(\mb{C})}c_1(\ov{\mc{L}}_1)\cdots c_1(\ov{\mc{L}}_d)=t\big(L_1\cdots L_d\big).\]
The reason for this formula is that the product of the Chern forms coincides with the cup product in cohomology. (See \cite[Chapter 9, Remark 8.10]{demailly_complex_nodate}.) By this formula, we have 
\begin{equation}\label{cruial}
	\begin{aligned}
		\big(\ov{\mc{L}}^k\cdot\ov{\mc{M}}(t)^{d+1-k}\big)&=\sum_{j=0}^{d+1-k}\binom{d+1-k}{j}\big(\ov{\mc{L}}^k\cdot\ov{\mc{M}}^{d+1-k-j}\cdot \ov{\mO}_{\mc{X}}(t)^{j}\big)\\	
		&=\big(\ov{\mc{L}}^k\cdot\ov{\mc{M}}^{d+1-k}\big)+(d+1-k)\big(\ov{\mc{L}}^k\cdot\ov{\mc{M}}^{d-k}\cdot \ov{\mO}_{\mc{X}}(t)\big)\\
		&=\big(\ov{\mc{L}}^k\cdot\ov{\mc{M}}^{d+1-k}\big)+(d+1-k)t\big(L^k\cdot M^{d-k}\big)
	\end{aligned}
\end{equation}
\[\]
\begin{lem}\label{Siu4}
	Let $\mathcal{X}\ra \Spe\mb{Z}$ be an arithmetic variety with relative dimension $d$.  Let $\ov{\mc{L}},\ov{\mc{M}}$ be nef Hermitian line bundles with generic fibers $L=\mc{L}_{\mb{Q}}$ and $M=\mc{M}_{\mb{Q}}$. Suppose $M$ is big, and $(L^k\cdot M^{d-k})>0$. Then there exists a positive constant $t_0>0$ only depends on $d, k, \ov{\mc{L}}$ and $\ov{\mc{M}}$,  such that for any arithmetic variety $\mc{Y}\ra \Spe \mb{Z}$ with $\dim\mc{Y}=\dim\mc{X}$ and any dominant generically finite morphism $g:\mc{Y}\ra \mc{X}$, we have
	\[ g^*\ov{\mc{L}}^k\leq
	(d+2-k)^k\frac{\big(L^k\cdot M^{d-k}\big)}{\big(M^{d}\big)}\cdot g^*\ov{\mc{M}}(t_0)^{k}.\]
	Here the symbol $\leq$ carries the same meaning as in Lemma \ref{Siu1}.
	
\end{lem}
\begin{proof}
	By applying Corollary \ref{Siu3} to $g^*\ov{\mc{L}}$ and $g^*\ov{\mc{M}}(t)$, we have
	\begin{equation}
		\begin{aligned}
			g^*\ov{\mc{L}}^k &\leq
			(d+2-k)^k\frac{\big(g^*\ov{\mc{L}}^k\cdot g^*\ov{\mc{M}}(t)^{d+1-k}\big)}{\big(g^*\ov{\mc{M}}(t)^{d+1}\big)}\cdot g^*\ov{\mc{M}}(t)^{k}\\
			&=
			(d+2-k)^k\frac{\big(\ov{\mc{L}}^k\cdot\ov{\mc{M}}(t)^{d+1-k}\big)}{\big(\ov{\mc{M}}(t)^{d+1}\big)}\cdot g^*\ov{\mc{M}}(t)^{k} \\
			&=(d+2-k)^k\frac{\big(\ov{\mc{L}}^k\cdot\ov{\mc{M}}^{d+1-k}\big)+(d+1-k)t\big(L^k\cdot M^{d-k}\big)}{\big(\ov{\mc{M}}^{d+1}\big)+(d+1)t\big(M^{d}\big)}\cdot g^*\ov{\mc{M}}(t)^{k}
		\end{aligned}
	\end{equation}
	We take $t_0$ sufficiently large such that
	\[\frac{\big(\ov{\mc{L}}^k\cdot\ov{\mc{M}}^{d+1-k}\big)+(d+1-k)t_0\big(L^k\cdot M^{d-k}\big)}{\big(\ov{\mc{M}}^{d+1}\big)+(d+1)t_0\big(M^{d}\big)}<\frac{\big(L^k\cdot M^{d-k}\big)}{\big(M^{d}\big)}.\]
	Thus we obtain
	\[\frac{\big(\ov{\mc{L}}^k\cdot\ov{\mc{M}}^{d+1-k}\big)+(d+1-k)t_0\big(L^k\cdot M^{d-k}\big)}{\big(\ov{\mc{M}}^{d+1}\big)+(d+1)t_0\big(M^{d}\big)}\cdot g^*\ov{\mc{M}}(t_0)^{k}\leq \frac{\big(L^k\cdot M^{d-k}\big)}{\big(M^{d}\big)}\cdot g^*\ov{\mc{M}}(t_0)^{k},\]
	which completes the proof.
\end{proof}

\begin{thm}\label{RelaDeg2}
	Let $f$ be a dominant rational self-map of X. Then for any integer $1\leq k\leq d$,
	one has
	\[\alpha_k(f)=\max\{\lambda_k(f),\lambda_{k-1}(f)\}.\]
\end{thm}

\begin{proof}
For the lower bound $\alpha_k(f)\geq \max\{\lambda_k(f),\lambda_{k-1}(f)\}$, \cite[Theorem B]{dang_higher_2022} proved it by arithmetic Bertini theorem. Here we will give a different proof.

We take an integral model $\mc{X}\ra \Spe \mb{Z}$ of $X$ and fix an ample Hermitian line bundle $\ov{\mc{H}}$ on $\mc{X}$. The rational self-map $f$ is lifted to a dominant rational self-map $F$ of $\mathcal{X}$. Let $\Gamma_F$ be the normalization of the graph of $F$ in $\mc{X}\times_{\Spe \mb{Z}}\mc{X}$, and $\pi_1,\pi_2$ be projections from $\Gamma_F$ onto the first and second component respectively. Fix a real number $t>0$.  By \cite[Theorem 2.2]{yuan_big_2008}
\[\pi_1^*\ov{\mO}_{\mc{X}}(t)\leq (d+1)\frac{\big(\pi_1^*\ov{\mO}_{\mc{X}}(t)\cdot \pi_1^*\ov{\mc{H}}^d\big)}{\big(\pi_1^*\ov{\mc{H}}^{d+1}\big)}\cdot\pi_1^*\ov{\mc{H}}=(d+1)\frac{t\cdot\big(H^d\big)}{\big(\ov{\mc{H}}^{d+1}\big)}\cdot\pi_1^*\ov{\mc{H}}.\]
(This is an equivalent formulation of the original theorem, and the proof follows the same arguments as in the proof of Corollary \ref{Siu1cor}.)
By intersecting both sides of the inequality with $\pi_2^*\ov{\mc{H}}^k\cdot \pi_1^*\ov{\mc{H}}^{d-k}$, we obtain
\begin{equation}
	\begin{aligned}
		\widehat{\deg}_k(F)&=\big(\pi_2^*\ov{\mc{H}}^k\cdot \pi_1^*\ov{\mc{H}}^{d+1-k}\big)\\
		&\geq \big(\pi_2^*\ov{\mc{H}}^k\cdot \pi_1^*\ov{\mc{H}}^{d-k}\cdot \pi_1^*\ov{\mO}_{\mc{X}}(t)\big)\cdot\frac{ \big(\ov{\mc{H}}^{d+1}\big)}{(d+1)\cdot t\cdot\big(H^d\big)}\\
		&=C\cdot\big(\pi_2^*H^k\cdot \pi_1^*H^{d-k}\big)\\
		&=C\cdot \deg_k(f),
	\end{aligned}
\end{equation}
where the constant $C=\frac{ \big(\ov{\mc{H}}^{d+1}\big)}{(d+1)(H^d)}$.

Similarly, we also have
\[\pi_2^*\ov{\mO}_{\mc{X}}(t)\leq (d+1)\frac{\big(\pi_2^*\ov{\mO}_{\mc{X}}(t)\cdot \pi_2^*\ov{\mc{H}}^d\big)}{\big(\pi_2^*\ov{\mc{H}}^{d+1}\big)}\cdot\pi_2^*\ov{\mc{H}}=(d+1)\frac{t\cdot\big(H^d\big)}{\big(\ov{\mc{H}}^{d+1}\big)}\cdot\pi_2^*\ov{\mc{H}}\]
by Yuan's inequality. If we intersect both sides of this inequality with $\pi_2^*\ov{\mc{H}}^{k-1}\cdot \pi_1^*\ov{\mc{H}}^{d+1-k}$, we obtain
\begin{equation}
	\begin{aligned}
		\widehat{\deg}_k(F)&=\big(\pi_2^*\ov{\mc{H}}^k\cdot \pi_1^*\ov{\mc{H}}^{d+1-k}\big)\\
		&\geq \big(\pi_2^*\ov{\mc{H}}^{k-1}\cdot \pi_1^*\ov{\mc{H}}^{d+1-k}\cdot \pi_2^*\ov{\mO}_{\mc{X}}(t)\big)\cdot\frac{ \big(\ov{\mc{H}}^{d+1}\big)}{(d+1)\cdot t\cdot\big(H^d\big)}\\
		&=C\cdot\big(\pi_2^*H^{k-1}\cdot \pi_1^*H^{d+1-k}\big)\\
		&=C\cdot \deg_{k-1}(f),
	\end{aligned}
\end{equation}
This implies $\alpha_k(f)\geq \max\{\lambda_k(f),\lambda_{k-1}(f)\}$.

Now we prove the converse inequality. Consider the same diagram as in Theorem \ref{Submulti2}.
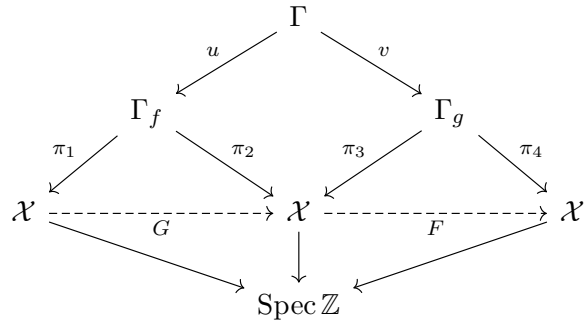
     \[\begin{tikzcd}
	&                                                   & \Gamma \arrow[ld, "u"'] \arrow[rd, "v"]   &                                                   &                    \\
	& \Gamma_f \arrow[ld, "\pi_1"'] \arrow[rd, "\pi_2"] &                                           & \Gamma_g \arrow[ld, "\pi_3"'] \arrow[rd, "\pi_4"] &                    \\
	\mc{X} \arrow[rrd] \arrow[rr, "G"', dashed] &                                                   & \mc{X} \arrow[d] \arrow[rr, "F"', dashed] &                                                   & \mc{X} \arrow[lld] \\
	&                                                   & \Spe \mb{Z}                              &                                                   &                   
\end{tikzcd}\]
For a fixed $F$, by Lemma \ref{Siu4}, there exists a constant $t_0>0$ such that for any $G$,
\[ v^*\pi_4^*\ov{\mc{H}}^k\leq
(d+2-k)^k\frac{\big(\pi_4^*H^k\cdot \pi_3^*H^{d-k}\big)}{\big(H^{d}\big)}\cdot v^*\pi_3^*\ov{\mc{H}}(t_0)^{k}.\]
Intersect the inequality above with $u^*\pi_1^*\ov{\mc{H}}^{d+1-k}$, then
\begin{equation}
\begin{aligned}
	\widehat{\deg}_k(F\circ G)=& \big(u^*\pi_1^*\ov{\mc{H}}^{d+1-k}\cdot v^*\pi_4^*\ov{\mc{H}}^k\big)\\
	\leq &(d+2-k)^k\frac{\big(\pi_4^*H^k\cdot \pi_3^*H^{d-k}\big)}{\big(H^{d}\big)}\cdot \big(u^*\pi_1^*\ov{\mc{H}}^{d+1-k}\cdot v^*\pi_3^*\ov{\mc{H}}(t_0)^{k}\big)\\
	=&C_1\deg_k(f) \big(\big(u^*\pi_1^*\ov{\mc{H}}^{d+1-k}\cdot v^*\pi_3^*\ov{\mc{H}}^k\big)+kt_0(u^*\pi_1^*H^{d+1-k}\cdot v^*\pi_3^*H^{k-1})\big)\\
	=& C_1\deg_k(f) \big(\big(u^*\pi_1^*\ov{\mc{H}}^{d+1-k}\cdot u^*\pi_2^*\ov{\mc{H}}^k\big)+kt_0(u^*\pi_1^*H^{d+1-k}\cdot u^*\pi_2^*H^{k-1})\big)\\
	=& C_1\deg_k(f) \big(\widehat{\deg}_k(G)+kt_0\deg_{k-1}(g)\big)
\end{aligned}
\end{equation}
where $C_1=\frac{(d+2-k)^k}{(H^{d})}$ is a constant that depends only on $d, k$ and $H$. The third line follows from (\ref{cruial}).

We also have the submultiplicity formula, which is proved in \cite[Theorem 1]{dang_degrees_2020},
\[\deg_{k-1}(f\circ g)\leq C_2 \deg_{k-1}(f)\cdot \deg_{k-1}(g).\]
The constant $C_2$ also depends only on $d, k$ and $H$.

We construct the matrices
\begin{equation}
U(F)=\begin{bmatrix} \widehat{\deg}_k(F)  \\ \deg_{k-1}(f)  \end{bmatrix},\quad M(f)=\begin{bmatrix} C_1\deg_k(f) & C_1k t_0 \deg_k(f) \\ 0 & C_2\deg_{k-1}(f)\end{bmatrix}.
\end{equation}
By setting $G=F^{n-1}$, we have
\[U(F^{n})\leq M(f)\cdot U(F^{n-1}),\]
then
\[U(F^{n})\leq M(f)^{n-1}\cdot U(F).\]
Here the inequality $v\leq v'$ between two vectors means $v_i\leq v'_{i}$ for for all coordinates $i$.

Consider the norm on $\mb{R}^2$ defined by
\[\|(v_1,v_2)^{\top}\|_{\infty}:=\max\{|v_1|,|v_2|\},\]
and for $M=(a_{ij})_{1\leq i,j\leq 2}\in M_2(\mb{R})$, define 
\[\|M\|_{\mr{max}}:=\max\{|a_{ij}|,1\leq i,j\leq 2\}.\]
With this notation, we observe that
\[\widehat{\deg}_{k}(F^n)\leq\|U(F^n)\|_{\infty}\leq \| M(f)^{n-1}\cdot U(F)\|_{\infty}\leq 2\|M(f)^{n-1}\|_{\mr{max}}\|U(F)\|_{\infty}.\]

We replace $F$ by $F^r$ in the inequality above and take $nr$-th root, then

\[\widehat{\deg}_{k}(F^{nr})^{1/nr}\leq 2^{1/nr}\|U(F^r)\|_{\infty}^{1/nr}\|M(f^r)^{n-1}\|_{\mr{max}}^{1/nr}.\]

Note that the eigenvalues of $M(f^r)$ are $C_1\deg_k(f^r)$ and $C_2\deg_{k-1}(f^r)$. Let $n\ra \infty$, we obtain

\[\alpha_k(f)=\lim_{n\ra\infty}\widehat{\deg}_{k}(F^{nr})^{1/nr}\leq \max\{C_1\deg_k(f^r),C_2\deg_{k-1}(f^r)\}^{1/r}.\]

Then, by letting $r\ra\infty$, we conclude the proof and obtain
\[\alpha_k(f)\leq \max\{\lambda_k(f),\lambda_{k-1}(f)\}.\]
\end{proof}

By directly applying Theorem \ref{RelaDeg2}, one can establish the birational invariance of $\alpha_k(f)$.

\begin{cor}\label{birinv}
Let $f:X\dashrightarrow X$ and $g:Y\dashrightarrow Y$ be rational self-maps of $X$ and $Y$. Let $\pi:X\dashrightarrow Y$ be a rational map with the diagram below.
\[\begin{tikzcd}
	X \arrow[r, "f", dashed] \arrow[d, "\pi", dashed] & X \arrow[d, "\pi", dashed] \\
	Y \arrow[r, "g", dashed]                          & Y                         
\end{tikzcd}.\]
Then for $1\leq k\leq \min\{\dim X, \dim Y\}$,
\begin{itemize}
	\item if $\pi$ is dominant, we have $\alpha_k(f)\geq \alpha_k(g)$;
	\item if $\pi$ is generically finite, we have $\alpha_k(f)\leq \alpha_k(g)$.
\end{itemize}
In particular, if $\pi$ is a generically finite dominant morphism, then $\alpha_k(f)=\alpha_k(g)$.
\end{cor}
\begin{proof}
We have the following classical result about dynamical degrees (see e.g., \cite[Lemma 2.4]{meng_kawaguchi-silverman_2022}). 
For $1\leq k\leq \min\{\dim X, \dim Y\}$,
\begin{itemize}
	\item if $\pi$ is dominant, then $\lambda_k(f)\geq \lambda_k(g)$;
	\item if $\pi$ is a generically finite morphism, then $\lambda_k(f)\leq \lambda_k(g)$.
\end{itemize}
By the formula $\alpha_k(f)=\max\{\lambda_k(f),\lambda_{k-1}(f)\}$, we get the conclusion.
\end{proof}

\section{Upper bounds of arithmetic degrees}\label{SecUpBd}
In this section, We will show that arithmetic degrees for subvarieties are independent of the choice of model $\mc{X}$ and Hermitian line bundle $\ov{\mc{H}}$. Additionally, we prove the inequality presented in Theorem \ref{UpBdd2}, which provides an upper bound for the arithmetic degrees of subvarieties. Our approach to proving this inequality closely follows the strategy of Xie in \cite{xie_remarks_2023}. Specifically, we first construct small sections such that the subvariety appears as an irreducible component of the intersection of these sections. Subsequently, we derive the inequality by applying an arithmetic analogue of \cite[Lemma 3.3]{jia_endomorphisms_2021}.

Let us begin by recalling the definition of arithmetic degrees for subvarieties.

Consider a normal projective variety $X$ over $\mathbb{Q}$ of dimension $d$, and let $f: X\dashrightarrow X$ be a dominant rational self-map of $X$. We assume that $\mathcal{X}$ is an integral model of $X$, which can be viewed as an arithmetic variety $\mathcal{X}\rightarrow \Spe \mathbb{Z}$ with generic fiber $X$. The rational self-map $f$ can be extended to a dominant rational self-map $F: \mathcal{X}\dashrightarrow \mathcal{X}$. Thus, we obtain a commutative diagram as described in Section \ref{Intro}.
 \[\begin{tikzcd}
	& \Gamma_F \arrow[ld, "\pi_1"'] \arrow[rd, "\pi_2"] &   \\
	\mc{X} \arrow[rr, "F", dashed] &                                                 & \mc{X}
\end{tikzcd}.\]
To define the arithmetic degree of subvarieties, we fix an ample Hermitian line bundle $\ov{\mc{H}}$ on $X$. Suppose $H=\mc{H}_{\mb{Q}}$. The definition proceeds as follows.

\begin{definition}\label{def2}
	Let $V$ be a $k$-dimensional irreducible subvariety of $X$ whose $f$-orbit is well-defined and $\mc{V}$ denote the Zariski closure of $V$ in $\mc{X}$. We define
	\[\deg_k(f,V)=\big(\pi_2^*H^k\cdot \pi_1^{\#}V\big)\]
	\[\widehat{\deg}_{k+1}(F,\mc{V})=\big(\pi_2^*\ov{\mc{H}}^{k+1}\cdot \pi_1^{\#}\mc{V}\big).\]
	Here the pull-back of a subvariety by a birational map means the strict transform.
	
	The \emph{geometric degree} of $(f,V)$ is defined as 
	\[\ov{\lambda}_k(f,V)=\limsup_{n\ra \infty}\deg_k(f^n,V)^{1/n},\]
	\[\underline{\lambda}_k(f,V)=\liminf_{n\ra \infty}\deg_k(f^n,V)^{1/n}.\]
	If $\ov{\lambda}_k(f,V)=\underline{\lambda}_k(f,V)$, we define 
	\[\lambda_k(f,V)=\ov{\lambda}_k(f,V)=\lim_{n\ra \infty}\deg_k(f^n,V)^{1/n}.\]
	
	The \emph{arithmetic degree} of $(f,V)$ is defined as 
	\[\ov{\alpha}_{k+1}(f,V)=\limsup_{n\ra\infty}\widehat{\deg}_{k+1}(F^n,\mc{V})^{1/n},\]
	\[\underline{\alpha}_{k+1}(f,V)=\liminf_{n\ra\infty}\widehat{\deg}_{k+1}(F^n,\mc{V})^{1/n}.\]
	If $\ov{\alpha}_{k+1}(f,V)=\underline{\alpha}_{k+1}(f,V)$, we define 
	\[\alpha_{k+1}(f,V)=\ov{\alpha}_{k+1}(f,V)=\lim_{n\ra\infty}\widehat{\deg}_{k+1}(F^n,\mc{V})^{1/n}.\]
\end{definition} 

The existence of the limit defining $\alpha_{k+1}(f, V)$ is a part of the high-dimensional analogue of Kawaguchi-Silverman conjecture, which will be presented in Section \ref{SecFur}. Our task now is to demonstrate the well-definedness of arithmetic degrees. Specifically, we aim to show that the arithmetic degrees of a subvariety are independent of the choice of integral models or the ample Hermitian line bundles.
\begin{lem}\label{Nakai}
	Let $\mc{X}\ra \Spe \mb{Z}$ be an arithmetic variety of relative dimension $d$ with generic fiber $X$. Then there exists a constant $C_1>0$ such that for any dominant rational self-map $f:X\dashrightarrow X$ and any subvariety $V$ of dimension $k$ of $X$ whose $f$-orbit is well-defined, we have
	\[\deg_k(f,V)\leq C_1 \cdot \widehat{\deg}_{k+1}(F,\mc{V}),\]
	where $F:\mc{X}\dashrightarrow \mc{X}$ is a rational self-map extending $f$ and $\mc{V}$ is the Zariski closure of $V$ in $\mc{X}$.
\end{lem}

\begin{proof}
	Consider an ample Hermitian line bundle $\ov{\mc{H}}$ on $\mc{X}$ with generic fiber $H$.  For a fixed positive number $t>0$, by applying the arithmetic Nakai-Moishezon theorem (see \cite[Corollary 4.8]{zhang_positive_1995} and \cite[Corollary 5.1]{moriwaki_semiample_2015}), we obtain a $\mb{Z}$-basis of $H^0(\mc{X},N\mc{H})$ contained in $\widehat{H}^0(\mc{X},N\ov{\mc{H}}+\ov{\mO}_{\mc{X}}(-t))$ for $N$ sufficiently large.

	Let $\Gamma_F$ denote the normalization of the graph of $F$, and $\pi_1$ and $\pi_2$ denote the projections onto the first and second coordinates, respectively. Denote $\pi_{2,*}\pi_1^{\#}\mc{V}$ by $\mc{Y}$, which is horizontal. Note that $H^0(X,NH)=H^0(\mc{X},N\mc{H})\otimes_{\mb{Z}}\mb{Q}$, and 
	sections whose support contains $\mc{Y}_{\mb{Q}}$ form a proper subspace of $H^0(X,NH)$ for sufficient large $N$. Consequently, we can select a small section $s\in \widehat{H}^0(\mc{X},N\ov{\mc{H}}+\ov{\mO}_{\mc{X}}(-t))$ from the $\mb{Z}$-basis such that $\di(s)$ and $\mc{Y}$ intersect properly. Then
	\[\big(\ov{\mc{H}}^{k}\cdot (N\ov{\mc{H}}+\ov{\mO}_{\mc{X}}(-t))\cdot \mc{Y}\big)=\big(\ov{\mc{H}}^{k}\cdot \di(s)\cdot \mc{Y}\big)-\int_{\mc{Y}(\mb{C})}\log\|s\|c_1(\ov{\mc{H}})^{k}\geq 0.\]
	Here $\|\cdot\|$ is the metric associated with $N\ov{\mc{H}}+\ov{\mO}_{\mc{X}}(-t)$ and the condition $\log \|s\|<0$ follows from $s$ being a small section, see Section \ref{Notation} for details. It follows that 
	\[N\widehat{\deg}_{k+1}(F,\mc{V})\geq -\big(\ov{\mc{H}}^{k}\cdot\ov{\mO}_{\mc{X}}(-t)\cdot \mc{Y}\big)=t\cdot \deg_k(f,V).\]
	We can take $C_1=\frac{N}{t}>0$, which completes the proof.
\end{proof}

\begin{prop}\label{welldefined}
	The arithmetic degrees $\ov{\alpha}_{k+1}(f,V)$ and $\underline{\alpha}_{k+1}(f,V)$ are well-defined. In other words, they are independent of the choice of the integral model or the ample Hermitian line bundle $\ov{\mc{H}}$. 
\end{prop}
\begin{proof}
	Let $\mc{X}_1$ and $\mc{X}_2$ be two integral models of $X$, and $\ov{\mc{H}}_1,\ov{\mc{H}}_2$ be ample Hermitian line bundles on $\mc{X}_1$ and $\mc{X}_2$ respectively. Consider the diagram below
	\[\begin{tikzcd}
		X \arrow[rdd, bend right] \arrow[rrd, bend left] \arrow[rd, "\Delta"] &                                                         &                    \\
		& \mc{X}_1\times_{\Spe\mb{Z}}\mc{X}_2 \arrow[d] \arrow[r] & \mc{X}_2 \arrow[d] \\
		& \mc{X}_1 \arrow[r]                                      & \Spe\mb{Z}        
	\end{tikzcd}.\]
	The map $\Delta$ is the diagonal embedding of $X$ in the generic fiber of $\mc{X}_1\times_{\Spe\mb{Z}}\mc{X}_2$, which is $X\times_{\Spe\mb{Q}}X$. Let $\mc{X}_3$ be the Zariski closure of $\Delta(X)$ in $\mc{X}_1\times_{\Spe\mb{Z}}\mc{X}_2$ with the reduced subscheme structure. Since $X$ is irreducible , $\mc{X}_3$ is irreducible. It follows that $\mc{X}_3$ is an integral model dominating both $\mc{X}_1$ and $\mc{X}_2$. Denote the projections from $\mc{X}_3$ to $\mc{X}_1$ and $\mc{X}_2$ by $p_1,p_2$, which are birational. The rational self-map $f$ of $X$ can be lifted to rational self-maps $F_i: \mc{X}_i\dashrightarrow \mc{X}_i, i=1,2,3$. Let $\Gamma_{F_i}$ be the normalization of the graph of $F_i$.
	
	By the universal properties of graphs, there are dominant rational maps $\wt{p}_1:\Gamma_{F_3}\dashrightarrow \Gamma_{F_1}$ and $\wt{p}_2:\Gamma_{F_3}\dashrightarrow \Gamma_{F_2}$. Moreover, $\wt{p}_1,\wt{p}_2$ are, in fact, morphisms. To see this, note that $\wt{p}_1$ is represented by a morphism between an open subset of $\Gamma_{F_3}$ and $\Gamma_{F_1}$, and this map is compatible with the morphism $p_1\times p_1: \mc{X}_3\times_{\Spe \mb{Z}}\mc{X}_3\ra \mc{X}_1\times_{\Spe \mb{Z}}\mc{X}_1$. Consequently, $p_1\times p_1$ maps $\Gamma_{F_3}$ into $\Gamma_{F_1}$, which implies that $\wt{p}_1$ is a morphism. A similar argument shows that $\wt{p}_2$ is also a morphism.

	Then we have the following diagram, which is commutative.
	\[\begin{tikzcd}
		& \Gamma_{F_1} \arrow[ld, "{\pi_{1,1}}"'] \arrow[rd, "{\pi_{1,2}}"]                                                &                                               \\
		\mc{X}_1 \arrow[rr, "F_1"' near end, dashed]                                      &                                                                                                                  & \mc{X}_1                                      \\
		& \Gamma_{F_3} \arrow[ld, "{\pi_{3,1}}"'] \arrow[rd, "{\pi_{3,2}}"] \arrow[uu, "\wt{p}_1" near start] \arrow[dd, "\wt{p}_2"' near end] &                                               \\
		\mc{X}_3 \arrow[uu, "p_1"] \arrow[dd, "p_2"'] \arrow[rr, "F_3"' near end, dashed] &                                                                                                                  & \mc{X}_3 \arrow[dd, "p_2"] \arrow[uu, "p_1"'] \\
		& \Gamma_{F_2} \arrow[ld, "{\pi_{2,1}}"'] \arrow[rd, "{\pi_{2,2}}"]                                                &                                               \\
		\mc{X}_2 \arrow[rr, "F_2"', dashed]                                      &                                                                                                                  & \mc{X}_2                                     
	\end{tikzcd}\]
	For a subvariety $V$ of $X$ of dimension $k$ which is not contained in the indeterminacy locus $I(f)$, denote the Zariski closure of $V$ in $\mc{X}_i$ by $\mc{V}_i$. Then by projection formula \cite[Proposition 2.3.1(iv)]{bost_heights_1994}, we have
\[\big(\pi_{3,2}^*p_1^*\ov{\mc{H}}_1^{k+1}\cdot \pi_{3,1}^{\#}\mc{V}_3\big)=\big(\wt{p}_1^*\pi_{1,2}^*\ov{\mc{H}}_1^{k+1}\cdot \pi_{3,1}^{\#}\mc{V}_3\big)=\big( \pi_{1,2}^*\ov{\mc{H}}_1^{k+1}\cdot \wt{p}_{1,*}\pi_{3,1}^{\#}\mc{V}_3\big).\]	

Since $\pi_{3,1}$ is a birational map, it induces an isomorphism between an open dense subset of $\Gamma_{F_1}$ and an open dense subset of $\mc{X}_1$ that contains $X\backslash I(f)$. A similar result holds for $\pi_{1,1}$ Moreover, $p_1$ and $\wt{p}_1$ are isomorphisms on their generic fibers. Since $V\not\subset I(f)$ and $\mc{V}_3$ is the Zariski closure of $V$ in $\mc{X}_3$, the generic point of $\mc{V}_3$ lies within the the open subset mentioned above. Therefore, we obtain the relation
\[\wt{p}_{1,*}\pi_{3,1}^{\#}\mc{V}_3=\pi_{1,1}^{\#}p_{1,*}\mc{V}_3.\]
It follows that
\[\big( \pi_{1,2}^*\ov{\mc{H}}_1^{k+1}\cdot \wt{p}_{1,*}\pi_{3,1}^{\#}\mc{V}_3\big)=\big( \pi_{1,2}^*\ov{\mc{H}}_1^{k+1}\cdot \pi_{1,1}^{\#}p_{1,*}\mc{V}_3\big)=\big( \pi_{1,2}^*\ov{\mc{H}}_1^{k+1}\cdot \pi_{1,1}^{\#}\mc{V}_1\big).\]
The last term is just $\widehat{\deg}_{k+1,\ov{\mc{H}}_1}(F_1,\mc{V}_1)$, which implies
	\begin{equation}\label{wd-51}
\big(\pi_{3,2}^*p_1^*\ov{\mc{H}}_1^{k+1}\cdot \pi_{3,1}^{\#}\mc{V}_3\big)=\widehat{\deg}_{k+1,\ov{\mc{H}}_1}(F_1,\mc{V}_1)
	\end{equation}
 Similarly, we have
	\begin{equation}\label{wd-52}
		\big(\pi_{3,2}^*p_2^*\ov{\mc{H}}_2^{k+1}\cdot \pi_{3,1}^{\#}\mc{V}_3\big)=\big( \pi_{2,2}^*\ov{\mc{H}}_2^{k+1}\cdot \pi_{2,1}^{\#}\mc{V}_2\big)=\widehat{\deg}_{k+1,\ov{\mc{H}}_2}(F_2,\mc{V}_2)
	\end{equation} 
	We take an integer $C$ sufficiently large such that $CH_1-H_2$ is ample and globally generated. Then we can take a $\mb{Q}$-basis $s_1,s_2,\ldots, s_r$ of $H^0(X,CH_1-H_2)$. Since
	\[H^0(\mc{X}_3,Cp_1^*\mc{H}_1-p_2^*\mc{H}_2)\otimes_{\mb{Z}}\mb{Q}=H^0(X,CH_1-H_2),\] then $s_i$ can be lifted to some $\wt{s}_i\in H^0(\mc{X}_3,Cp_1^*\mc{H}_1-p_2^*\mc{H}_2)$ after replacing $s_i$ by $m\cdot s_i$ for some integer $m$.  Consider a positive number $t$ such that $e^t>\max\limits_{1\leq i\leq n}\{\|\wt{s}_i\|_{\mr{sup}}\}$, then 
	\[\wt{s}_i\in \widehat{H}^0(\mc{X}_3,(Cp_1^*\ov{\mc{H}}_1-p_2^*\ov{\mc{H}}_2)(t))\]
	for all $1\leq i\leq r$. We denote the metric on $(Cp_1^*\ov{\mc{H}}_1-p_2^*\ov{\mc{H}}_2)(t)$ by $\|\cdot\|_t=e^{-t}\|\cdot\|$. It follows that $\|\wt{s}_i\|_t<1, i=1,2,\ldots,r$. 
	
	We have
	\begin{equation}
		\begin{aligned}
			&\big(\pi_{3,2}^*(C^{k+1}p_1^*\ov{\mc{H}}_1^{k+1}-p_2^*\ov{\mc{H}}_2^{k+1})\cdot \pi_{3,1}^{\#}\mc{V}_3\big)\\
			=&\big((C^{k+1}p_1^*\ov{\mc{H}}_1^{k+1}-p_2^*\ov{\mc{H}}_2^{k+1})\cdot \pi_{3,2,*}\pi_{3,1}^{\#}\mc{V}_3\big)\\
			=&\sum_{i=0}^{k}\big(C^{k-i}p_1^*\ov{\mc{H}}_1^{k-i}\cdot p_2^*\ov{\mc{H}}_2^i\cdot (Cp_1^*\ov{\mc{H}}_1-p_2^*\ov{\mc{H}}_2)\cdot \pi_{3,2,*}\pi_{3,1}^{\#}\mc{V}_3\big)\\
			=&\sum_{i=0}^{k}\big(C^{k-i}p_1^*\ov{\mc{H}}_1^{k-i}\cdot p_2^*\ov{\mc{H}}_2^i\cdot (Cp_1^*\ov{\mc{H}}_1-p_2^*\ov{\mc{H}}_2)(t)\cdot \pi_{3,2,*}\pi_{3,1}^{\#}\mc{V}_3\big)-t\sum_{i=0}^{k}\big(C^{k-i}H_1^{k-i}\cdot H_2^i\cdot \pi_{2,*}\pi_{1}^{\#}V\big)
		\end{aligned}
	\end{equation}
	The reason for the last equality is that $(p_1^*\mc{H}_1)_{\mb{Q}}=H_1, (p_2^*\mc{H}_2)_{\mb{Q}}=H_2$, and
	\[\big(C^{k-i}p_1^*\ov{\mc{H}}_1^{k-i}\cdot p_2^*\ov{\mc{H}}_2^i\cdot \ov{\mc{O}}_{\mc{X}_3}(t)\cdot \pi_{3,2,*}\pi_{3,1}^{\#}\mc{V}_3\big)=t\cdot\big(C^{k-i}H_1^{k-i}\cdot H_2^i\cdot \pi_{2,*}\pi_{1}^{\#}V\big)\]
	
	The cycle $\pi_{2,*}\pi_1^{\#}V$ is irreducible, and $CH_1-H_2$ is globally generated. Consequently, the global sections of $CH_1-H_2$ whose support contains $\pi_{2,*}\pi_1^{\#}V$ form an proper subspace of $H^0(X,CH_1-H_2)$. Thus, there exists a section $s_j$ whose support does not contain $\pi_{2,*}\pi_1^{\#}V$. Denote the cycle $\pi_{3,2,*}\pi_{3,1}^{\#}\mc{V}_3$ by $\mc{Y}$,  taking multiplicities into account. Then $\di(\wt{s}_j)$ and $\mc{Y}$ intersect properly. We have
	\begin{equation}
		\begin{aligned}
			&\big(C^{k-i}p_1^*\ov{\mc{H}}_1^{k-i}\cdot p_2^*\ov{\mc{H}}_2^i\cdot (Cp_1^*\ov{\mc{H}}_1-p_2^*\ov{\mc{H}}_2)(t)\cdot \pi_{3,2,*}\pi_{3,1}^{\#}\mc{V}_3\big)\\
			=&\big(C^{k-i}p_1^*\ov{\mc{H}}_1^{k-i}\cdot p_2^*\ov{\mc{H}}_2^i\cdot \di(\wt{s}_j)\cdot \mc{Y}\big)-\int_{\mc{Y}(\mb{C})}\log\|\wt{s}_j\|_tC^{k-i}c_1(p_1^*\ov{\mc{H}}_1)^{k-i}\cdot c_1(p_2^*\ov{\mc{H}}_2)^i\\
			\geq & 0
		\end{aligned}
	\end{equation}
	The reason for the last line is that $\wt{s}_j$ is a small section of $(Cp_1^*\ov{\mc{H}}_1-p_2^*\ov{\mc{H}}_2)(t)$. We also have 
	\[\big(C^{k-i}H_1^{k-i}\cdot H_2^i\cdot \pi_{2,*}\pi_{1}^{\#}V\big)\leq C^k\big(H_1^k\cdot \pi_{2,*}\pi_{1}^{\#}V\big)\]
	since $CH_1-H_2$ is ample. Thus 
	\[\big(\pi_{3,2}^*(C^{k+1}p_1^*\ov{\mc{H}}_1^{k+1}-p_2^*\ov{\mc{H}}_2^{k+1})\cdot \pi_{3,1}^{\#}\mc{V}_3\big)\geq -t\sum_{i=0}^{k}\big(C^{k-i}H_1^{k-i}\cdot H_2^i\cdot \pi_{2,*}\pi_{1}^{\#}V\big)\geq -t(k+1)C^k\big(H_1^k\cdot \pi_{2,*}\pi_{1}^{\#}V\big).\]
	Combining this inequality with (\ref{wd-51}) and (\ref{wd-52}), we have 
	\begin{equation}
		\begin{aligned}
			\widehat{\deg}_{k+1,\ov{\mc{H}}_2}(F_2,\mc{V}_2)&\leq C^{k+1}\widehat{\deg}_{k+1,\ov{\mc{H}}_1}(F_1,\mc{V}_1)+t(k+1)C^k\deg_{k,H_1}(f,V)\\
			&\leq C^k\big(C+t(k+1)C_1\big)\cdot \widehat{\deg}_{k+1,\ov{\mc{H}}_1}(F_1,\mc{V}_1).
		\end{aligned}
	\end{equation}
 	The second inequality is obtained by Lemma \ref{Nakai}. 
 	
 	Replacing the rational map $f$ by $f^n$, then
			\[\widehat{\deg}_{k+1,\ov{\mc{H}}_2}(F_2^n,\mc{V}_2)^{1/n}\leq C^{k/n}\big(C+t(k+1)C_1\big)^{1/n}\cdot \widehat{\deg}_{k+1,\ov{\mc{H}}_1}(F_1^n,\mc{V}_1)^{1/n}.\]
	Let $n\ra \infty$, we have
	\[\limsup_{n\ra \infty}\widehat{\deg}_{k+1,\ov{\mc{H}}_2}(F_2^n,\mc{V}_2)^{1/n}\leq \limsup_{n\ra \infty}\widehat{\deg}_{k+1,\ov{\mc{H}}_1}(F_1^n,\mc{V}_1)^{1/n}.\]
	Similarly, we have
	\[\limsup_{n\ra \infty}\widehat{\deg}_{k+1,\ov{\mc{H}}_1}(F_1^n,\mc{V}_1)^{1/n}\leq \limsup_{n\ra \infty}\widehat{\deg}_{k+1,\ov{\mc{H}}_2}(F_2^n,\mc{V}_2)^{1/n}.\]
	
	This implies the arithmetic degree $\ov{\alpha}_{k+1}(f,V)$ is well-defined. By the same approach, the arithmetic degree $\underline{\alpha}_{k+1}(f,V)$ is well-defined. 
\end{proof}

Let $L$ be an ample line bundle on $X$, and let $\mc{X}$ be an integral model of $X$. Suppose $\ov{\mc{L}}$ is a nef Hermitian line bundle that extends $L$. Then the height of a subvariety $V$ is defined by the following formula.
\[h_{\ov{\mc{L}}}(V)=\frac{\big(\ov{\mc{L}}^{\dim V+1}\cdot \mc{V}\big)}{(\dim V+1)(L^{\dim V}\cdot V)},\]
where $\mc{V}$ is the Zariski closure of $V$ in $\mc{X}$. 

This height function is a higher dimensional generalization of the classical height function for points. Recall that in \cite[Proposition 3]{kawaguchi_dynamical_2016-1}, Kawaguchi and Silverman established a formula for the growth rate of the height counting function for points in orbits in terms of arithmetic degree. Here, we extend this elementary result to the higher dimensional case. The proof closely follows the approach taken in their paper.

\begin{prop}
	Let $V$ be an irreducible subvariety of dimension $k$ of $X$ whose $f$-orbit is well-defined and infinite. Assume the geometric degree $\lambda_{k}(f,V)$ and arithmetic degree $\alpha_{k+1}(f,V)$ exist, then
	\[\lim_{T\ra \infty}\frac{\{W\in \mO_f(V): h_{\ov{\mc{L}}}(W)\leq T\}}{\log T}=\frac{1}{\log \alpha_{k+1}(f,V)-\log \lambda_{k}(f,V)}.\]
	In particular, when $\alpha_{k+1}(f,V)=\lambda_{k}(f,V)$, the left side converges to $+\infty$.
\end{prop}

\begin{proof}
	For the height of a subvariety, we have
	\[h_{\ov{\mc{L}}}(f^n(V))=\frac{\big(\ov{\mc{L}}^{k+1}\cdot \pi_2(\pi_1^{\#}\mc{V})\big)}{(k+1)\big(L^k\cdot \pi_2(\pi_1^{\#}V)\big)}=\frac{\big(\ov{\mc{L}}^{k+1}\cdot \pi_{2,*}\pi_1^{\#}\mc{V}\big)}{(k+1)\big(L^k\cdot \pi_{2,*}\pi_1^{\#}V\big)}=\frac{\big(\pi_2^*\ov{\mc{L}}^{k+1}\cdot \pi_1^{\#}\mc{V}\big)}{(k+1)\big(\pi_2^*L^k\cdot\pi_1^{\#}V\big)}.\]
	Then we have the following formula. 
	\[\lim_{n\ra\infty}h_{\ov{\mc{L}}}(f^n(V))^{1/n}=\lim_{n\ra\infty}\frac{\widehat{\deg}_{k+1}(F^n,\mc{V})^{1/n}}{\deg_{k}(f^n,V)^{1/n}}=\frac{\alpha_{k+1}(f,V)}{\lambda_k(f,V)}.\]
	For every $\varepsilon>0$, there is an $n_0(\varepsilon)$ such that
	\begin{equation}\label{equbd}
		(1-\varepsilon)\frac{\alpha_{k+1}(f,V)}{\lambda_k(f,V)}\leq h_{\ov{\mc{L}}}(f^n(V))^{1/n}\leq (1+\varepsilon)\frac{\alpha_{k+1}(f,V)}{\lambda_k(f,V)}
	\end{equation} 
	for all $n\geq n_0(\varepsilon)$. This implies 
	\[\left\{n\geq n_0 : (1+\varepsilon)\frac{\alpha_{k+1}(f,V)}{\lambda_k(f,V)}\leq T^{1/n}\right\}\subset \left\{n\geq n_0 : h_{\ov{\mc{L}}}(f^n(V))\leq T\right\}.\]
	By counting the number of elements in these sets, we have 
	\[ \frac{\log T}{\log(1+\varepsilon)+\log\alpha_{k+1}(f,V)-\log \lambda_k(f,V)}\leq\#\left\{n\geq 0 : h_{\ov{\mc{L}}}(f^n(V))\leq T\right\}+n_0(\varepsilon)+1.\]
	Dividing by $\log T$ and taking $T\ra \infty$, $\varepsilon\ra 0$ we obtain
	\[\frac{1}{\log\alpha_{k+1}(f,V)-\log \lambda_k(f,V)}\leq \liminf_{T\ra \infty}\frac{\{W\in \mO_f(V): h_{\ov{\mc{L}}}(W)\leq T\}}{\log T}.\]
	If $\alpha_{k+1}(f,V)=\lambda_k(f,V)$, the left side is $+\infty$, then we get the conclusion. Now we assume $\alpha_{k+1}(f,V)>\lambda_k(f,V)$. By (\ref{equbd}), there is an inclusion.
	\[\left\{n\geq n_0 : h_{\ov{\mc{L}}}(f^n(V))\leq T\right\}\subset\left\{n\geq n_0 : (1-\varepsilon)\frac{\alpha_{k+1}(f,V)}{\lambda_k(f,V)}\leq T^{1/n}\right\}.\]
	Then we have
	\[\#\left\{n\geq 0 : h_{\ov{\mc{L}}}(f^n(V))\leq T\right\}\leq \frac{\log T}{\log(1-\varepsilon)+\log\alpha_{k+1}(f,V)-\log \lambda_k(f,V)}+n_0(\varepsilon)+1.\]
	Dividing by $\log T$ and taking $T\ra \infty$, $\varepsilon\ra 0$ we obtain
	\[\limsup_{T\ra \infty}\frac{\{W\in \mO_f(V): h_{\ov{\mc{L}}}(W)\leq T\}}{\log T}\leq \frac{1}{\log\alpha_{k+1}(f,V)-\log \lambda_k(f,V)}.\]
	Hence
	\[\lim_{T\ra \infty}\frac{\{W\in \mO_f(V): h_{\ov{\mc{L}}}(W)\leq T\}}{\log T}=\frac{1}{\log\alpha_{k+1}(f,V)-\log \lambda_k(f,V)}.\]
\end{proof}

Then we will prove a lemma, which is an arithmetic analogue of \cite[Lemma 3.3]{jia_endomorphisms_2021}. This lemma is crucial in the proof of Theorem \ref{UpBdd2}.

\begin{lem}\label{LemPos}
	Let $\mc{X}\ra \Spe\mb{Z}$ be an arithmetic variety of relative dimension $d$. Suppose Hermitian line bundles $\ov{\mc{L}}_1,\ldots,\ov{\mc{L}}_r$ are nef and $\mc{V}$ is a horizontal irreducible subvariety of $\mc{X}$ of codimension $r$. Suppose there exists $r$ small sections $s_1\in \widehat{H}_0(\mc{X},\ov{\mc{L}}_1),\ldots,s_r\in \widehat{H}_0(\mc{X},\ov{\mc{L}}_r)$ such that $\mc{V}$ is an irreducible component of $\di(s_1)\cap \cdots \cap \di(s_r)$, then for any $d+1-r$ nef Hermitian line bundles $\ov{\mc{H}}_1, \ldots, \ov{\mc{H}}_{d+1-r}$, we have 
	\[\big(\ov{\mc{H}}_1 \cdots \ov{\mc{H}}_{d+1-r}\cdot\ov{\mc{L}}_1\cdots\ov{\mc{L}}_r\big)\geq \big(\ov{\mc{H}}_1 \cdots \ov{\mc{H}}_{d+1-r}\cdot\mc{V}\big).\]
\end{lem}
\begin{proof}
	We will prove this lemma by induction. For $r=1$, $\mc{V}$ is an irreducible component of $\di(s_1)$.
	Fix $d$ nef Hermitian line bundles $\ov{\mc{H}}_1,\ldots,\ov{\mc{H}}_{d}$
	, then
	\[\big(\ov{\mc{H}}_1\cdots\ov{\mc{H}}_{d}\cdot \ov{\mc{L}}_1\big)=\ov{\mc{H}}_1\cdots\ov{\mc{H}}_{d}\cdot \di(s_1)-\int_{\mc{X}(\mb{C})}\log\|s_1\|c_1(\ov{\mc{H}}_1)\cdots c_1(\ov{\mc{H}}_d)\geq \ov{\mc{H}}_1\cdots\ov{\mc{H}}_{d}\cdot \mc{V}.\]
	Assume the result holds for $r-1$ sections,  it suffices to prove 
	\[\big(\ov{\mc{H}}_1\cdots\ov{\mc{H}}_{d+1-r}\cdot\ov{\mc{L}}_1\cdots\ov{\mc{L}}_r\big)\geq \big(\ov{\mc{H}}_1\cdots\ov{\mc{H}}_{d+1-r}\cdot \mc{V}\big)\]
	for any fixed nef Hermitian line bundles $\ov{\mc{H}}_1,\ldots,\ov{\mc{H}}_{d+1-r}$. 
	Indeed, after relabeling, there exists an irreducible component $\mc{W}$ of $\bigcap\limits_{i=1}^{r-1}\di(s_i)$ of codimension $r-1$ such that $\mc{W}\not\subset \di(s_r)$ and $\mc{V}\subset \mc{W}\cap \di(s_r)$ is an irreducible component of it. Then by the induction hypothesis
	\begin{equation}
		\begin{aligned}
			\big(\ov{\mc{H}}_1\cdots\ov{\mc{H}}_{d+1-r}\cdot\ov{\mc{L}}_1\cdots\ov{\mc{L}}_r\big)&\geq \big(\ov{\mc{H}}_1\cdots\ov{\mc{H}}_{d+1-r}\cdot\ov{\mc{L}}_r\cdot \mc{W}\big)\\
			&=\big(\ov{\mc{H}}_1\cdots\ov{\mc{H}}_{d+1-r}\cdot \mc{W}\cdot \di(s_r)\big)-\int_{\mc{W}(\mb{C})}\log\|s_r\|c_1(\ov{\mc{H}}_1)\cdots c_1(\ov{\mc{H}}_{d+1-r})\\
			&\geq \big(\ov{\mc{H}}_1\cdots\ov{\mc{H}}_{d+1-r}\cdot \mc{V}\big),
		\end{aligned}
	\end{equation}
	which completes the proof.
\end{proof}

\begin{thm}\label{UpBdd2}
	Let $f: X \dashrightarrow X$ be a dominant rational map, and let $V$ be a $k$-dimensional irreducible subvariety of $X$ whose $f$-orbit is well-defined. Then
	\[\overline{\alpha}_{k+1}(f,V)\leq \alpha_{k+1}(f).\]
\end{thm}

\begin{proof}
	
	Firstly, we take an integral model $\mc{X}\ra\Spe\mb{Z}$ of $X$ and fix an ample Hermitian line bundle $\ov{\mc{H}}$ on $\mc{X}$, whose restriction on $X$ is denoted by $H$. The rational self-map $f$ is lifted to $F$, which is a dominant rational self-map of $\mc{X}$. Let $I$ be the ideal sheaf of $V$ on $X$. We can take $r$ sufficiently large such that $H^{\otimes r}\otimes I$ is globally generated. Then we can replace $\ov{\mc{H}}$ by $\ov{\mc{H}}^{\otimes r}$. We take global sections $s_1,s_2,\ldots,s_{d-k}\in H^0(X,H\otimes I)\subset H^0(X,H)$ such that $\dim(\di(s_1)\cap \cdots \cap \di(s_{d-k}))=k$. Note that
	\[H^0(X,H)=H^0(\mc{X},\mc{H})\otimes_{\mb{Z}}\mb{Q},\]
	then after replacing $s_i$ by $m\cdot s_i$ for some integer $m$ sufficiently large, we can assume that all sections $s_1,\ldots,s_{d-k}$ can be lifted to $\wt{s}_1,\ldots,\wt{s}_{d-k}\in H^0(\mc{X},\mc{H}).$
	The cycle $\di(\wt{s}_1)\cap \cdots \cap \di(\wt{s}_{d-k})$ contains $\mc{V}$ as an irreducible component, where $\mc{V}$ is the Zariski closure of $V$ in $\mc{X}$. We can replace $\ov{\mc{H}}$ by $\ov{\mc{H}}(c)$ for $c$ sufficiently large such that $\wt{s}_1,\ldots,\wt{s}_{d-k}$ are small sections of $\ov{\mc{H}}$.
	
	Now, for any $\varepsilon>0$, there exists $C>0$ such that $\widehat{\deg}_{k+1}(F^n)\leq C(\alpha_{k+1}(f)+\varepsilon)^n$ for all $n>0$.
	Consider the normalization of the graph of $F^n$, one get the following diagram:
	\[\begin{tikzcd}
		& \Gamma \arrow[ld, "\pi_1"'] \arrow[rd, "\pi_2"] &   \\
		\mc{X} \arrow[rr, "F^n", dashed] &                                                 & \mc{X}
	\end{tikzcd}.\]
	The strict transformation $\pi_1^{\#}\mc{V}$ is still an irreducible component of $\di(\pi_1^*\wt{s}_1)\cap \cdots \cap \di(\pi_1^*\wt{s}_{d-k})$. Then by Lemma \ref{LemPos},
	\begin{equation}\label{Upbd4.4}
		\begin{aligned}
			\widehat{\deg}_{k+1}(F^n,\mc{V})=\big(\pi_2^*\ov{\mc{H}}^{k+1}\cdot \pi_1^{\#}\mc{V}\big)\leq \big(\pi_2^*\ov{\mc{H}}^{k+1}\cdot \pi_1^*\ov{\mc{H}}^{d-k}\big)\leq C(\alpha_{k+1}(f)+\varepsilon)^n.
		\end{aligned}
	\end{equation}
	Taking $n$-th root in both side of (\ref{Upbd4.4}) and letting $n\ra \infty, \varepsilon\ra 0$, we get the upper bound
	\[\overline{\alpha}_{k+1}(f,V)=\limsup_{n\ra \infty}	\widehat{\deg}_{k+1}(F^n,\mc{V})^{1/n}\leq \alpha_{k+1}(f).\]
\end{proof}

\section{A generalized version of Kawaguchi-Silverman conjecture and its counterexample}\label{SecFur}
In \cite[Conjecture 1.6]{dang_higher_2022}, the authors proposed a conjecture about higher arithmetic degrees, which is a generalized version of Kawaguchi-Silverman conjecture for points. 
\begin{conj}\label{Kawa-Sil}\cite[Conjecture 1.6, Conjecture 1.7]{dang_higher_2022}
Let $X$ be a projective variety over a number field $K$ and $f:X\dashrightarrow X$ be a dominant rational self-map, and let $V$ be an irreducible subvariety of dimension $k$. Suppose the $f$-orbit of $V$ is well-defined. Then
\begin{enumerate}
	\item The limit defining $\alpha_{k+1}(f,V)$ exists.
	\item If $V$ has Zariski dense orbit, then $\alpha_{k+1}(f,V)=\alpha_{k+1}(f)$.
\end{enumerate}	
\end{conj}
Unfortunately,  Conjecture \ref{Kawa-Sil}(2) is false. In fact, we will give a counterexample here, which is derived from the counterexample to the old version of dynamical Manin-Mumford conjecture in \cite{ghioca_towards_2011}.

In order to present the counterexample, we first need to introduce the concept of adelic line bundles. The theory of adelic line bundles on projective varieties was developed by Zhang in \cite{zhang_positive_1995} and \cite{zhang_small_1995}. Recently, in \cite{yuan_adelic_2021}, Yuan and Zhang extended this theory to quasi-projective varieties. In what follows, we provide the necessary definitions and properties that we will use.
\begin{definition}
	Let $X$ be a projective variety over a number field $K$ and $L$ be a line bundle on $X$. An \emph{adelic metric} on $L$ is a \emph{coherent} collection $\{\|\cdot\|_v\}_v$ of bounded $K_v$-metrics $\|\cdot\|_v$ on $L_{K_v}$ over $X_{K_v}$ over all places $v$ of $K$. 
	
	The collection
	$\{\|\cdot\|_v\}_v$ is \emph{coherent} means that, there exists a finite set $S$ of non-archimedean
	places of $K$ and a (projective and flat) integral model $(\mc{X},\mc{L})$ of $(X, L)$ over
	$\Spe O_K - S$, such that the $K_v$-norm $\|\cdot\|_v$ is induced by $(\mc{X}_{O_{K_v}}, \mc{L}_{O_{K_v}})$ for
	all $v\in \Spe O_K - S$.
	
	We write $L=(L,\{\|\cdot\|_v\}_v)$ and call it an \emph{adelic line bundle} on $X$.
\end{definition}
In particular, suppose $(\mc{X},\ov{\mc{L}})$ is an arithmetic model of $(X,L)$ over $\Spe O_K$. Then the Hermitian line bundle $\ov{\mc{L}}$ induces an adelic line bundle $\ov{L}$ with generic fiber $L$.

Let $\ov{L}$ be a nef adelic line bundle whose generic fiber $L$ is ample. We define the height of a subvariety $V$ of $X$ by the following formula.
\[h_{\ov{L}}(V)=\frac{\big((\ov{L}|_V)^{\dim V+1}\big)}{(\dim V+1)(L^{\dim V}\cdot V)}.\]
Suppose $\ov{L}$ and $\ov{L}'$ are two nef adelic line bundles with the same generic fiber $L$, then there exists a constant $C>0$ such that for any subvariety $V$, we have
\begin{equation}\label{htbd}
|h_{\ov{L}}(V)-h_{\ov{L}'}(V)|\leq C.
\end{equation}

In fact, in order to compute the arithmetic degree, it is sufficient to use a nef adelic line bundle whose generic fiber is ample, rather than an ample Hermitian line bundle. To be specific, let $\ov{L}$ be a nef adelic line bundle on $X$ with an ample generic fiber $L$. We can define
\[\widehat{\deg}_{k+1,\ov{L}}(f,V)=\big(\pi_2^*\ov{L}^{k+1}\cdot \pi_1^{\#}V\big),\]
where $\pi_1,\pi_2$ are projections from $\Gamma_f$ onto the first and second component respectively.

Then we have the following lemma.
\begin{lem}\label{nefcompute}
	let $\ov{L}$ be a nef adelic line bundle on $X$ with an ample generic fiber $L$. Suppose $V$ is a $k$-dimensional  irreducible subvariety whose $f$-orbit is well-defined. Then the arithmetic degree can be computed by formulas
	\[\ov{\alpha}_{k+1}(f,V)=\limsup_{n\ra \infty}\max\left\{\widehat{\deg}_{k+1,\ov{L}}(f^n,V)^{1/n},\deg_{k,L}(f^n,V)^{1/n}\right\}\]
	and
	\[\underline{\alpha}_{k+1}(f,V)=\liminf_{n\ra \infty}\max\left\{\widehat{\deg}_{k+1,\ov{L}}(f^n,V)^{1/n},\deg_{k,L}(f^n,V)^{1/n}\right\}\]
\end{lem}
\begin{proof}
We can replace $L$ by $rL$ for some integer $r$ such that there exists an arithmetic model $(\mc{X},\ov{\mc{L}})$ of $(X,L)$, where $\ov{\mc{L}}$ is an ample Hermitian line bundle. The rational self-map $f$ can be lifted to a rational self-map $F$ of $\mc{X}$. Let $\Gamma_F$ be the normalization of the graph of $F$ in $\mc{X}\times_{\Spe \mb{Z}}\mc{X}$, and $\pi_1,\pi_2$ be projections from $\Gamma_F$ onto the first and second component respectively.
By (\ref{htbd}), there exists a constant $C$ such that for any $f$,
\[\big|\big(\ov{L}^{k+1}\cdot \pi_{2,*}\pi_1^{\#}V\big)-\big(\ov{\mc{L}}^{k+1}\cdot \pi_{2,*}\pi_1^{\#}\mc{V}\big)\big|\leq C(k+1)\big(L^k\cdot \pi_{2,*}\pi_1^{\#}V\big).\]
Thus
\[ \widehat{\deg}_{k+1,\ov{\mc{L}}}(F,\mc{V})\leq \widehat{\deg}_{k+1,\ov{L}}(f,V)+C(k+1)\deg_{k,L}(f,V).\]
Replacing $f$ by $f^n$ and taking $n$-th root, we get
\[\ov{\alpha}_{k+1}(f,V)\leq \limsup_{n\ra \infty}\max\left\{\widehat{\deg}_{k+1,\ov{L}}(f^n,V)^{1/n},\deg_{k,L}(f^n,V)^{1/n}\right\}.\]
On the other hand,
\[\widehat{\deg}_{k+1,\ov{L}}(f,V)\leq \widehat{\deg}_{k+1,\ov{\mc{L}}}(F,\mc{V})+C(k+1)\deg_{k,L}(f,V).\]
According to Lemma \ref{Nakai}, there exists a constant $C_1$ such that
\[\deg_{k,L}(f,V)\leq C_1\cdot \widehat{\deg}_{k+1,\ov{\mc{L}}}(F,\mc{V}).\]
Then we have
\[\limsup_{n\ra \infty}\widehat{\deg}_{k+1,\ov{L}}(f^n,V)^{1/n}\leq \ov{\alpha}_{k+1}(f,V)\]
and
\[ \limsup_{n\ra \infty}\deg_{k,L}(f^n,V)^{1/n}\leq \ov{\alpha}_{k+1}(f,V).\]
This implies
\[\ov{\alpha}_{k+1}(f,V)=\limsup_{n\ra \infty}\max\left\{\widehat{\deg}_{k+1,\ov{L}}(f^n,V)^{1/n},\deg_{k,L}(f^n,V)^{1/n}\right\}.\]
Using the same approach, one can establish the formula for $\underline{\alpha}_{k+1}(f,V)$.
\end{proof}

\begin{remark}
	There exist examples where \[\lim_{n\ra\infty}\widehat{\deg}_{k+1,\ov{L}}(f^n,V)^{1/n}<\lim_{n\ra\infty}\deg_{k,L}(f^n,V)^{1/n}.\]
	Therefore, in Definition \ref{def2}, we cannot use a nef adelic line bundle $\ov{L}$ with an ample generic fiber to define arithmetic degrees.
\end{remark}

Now let $f$ be a polarized endomorphism on $X$, that is, an ample line bundle $L$ exists such that such that $f^*L=qL$ for some integer $q\geq 2$. In \cite{zhang_small_1995}, by Tate's limiting argument, Zhang constructed a nef adelic line bundle $\ov{L}_f$ extending $L$, such that $f^*\ov{L}_f=q\ov{L}_f$. The height $h_{\ov{L}_f}$ is then called the \emph{canonical height}, which shares many properties with the classical canonical height of points.

When $f$ is a polarized endomorphism, arithmetic degrees can be easily computed.
\begin{prop}\label{egpolarized}
	Let $X$ be a projective variety over a number field $K$ of dimension $d$ and $L$ be an ample line bundle on $X$. Suppose $f:X\ra X$ is a polarized endomorphism with $f^*L=qL$, $q>1$. Then the arithmetic degrees are given by
	\[\alpha_{k+1}(f)=q^{k+1}\]
	and
	\[\alpha_{k+1}(f,V) = \begin{cases}
		q^{k+1}, & \text{if } h_{\ov{L}_f}(V)>0 \\
		q^k, & \text{if } h_{\ov{L}_f}(V)=0.
	\end{cases}\]
In particular, $\alpha_{k+1}(f,V)=\alpha_k(f)$ if and only if $h_{\ov{L}_f}(V)>0$.
\end{prop}
\begin{proof}
For any integer $0\leq k\leq d$, the dynamical degree can be computed by 
\[\lambda_k(f)=\lim_{n\ra \infty}((f^n)^*L^k\cdot L^{d-k})^{1/n}=q^k\lim_{n\ra \infty}(L^d)^{1/n}=q^k.\]
Hence the arithmetic degree can be computed by 
\[\alpha_{k+1}(f)=\max\{\lambda_k(f),\lambda_{k+1}(f)\}=q^{k+1}.\]
In addition, we have
\[\widehat{\deg}_{k+1,\ov{L}_f}(f^n,V)=\big((f^n)^*\ov{L}_f^{k+1}\cdot V\big)=q^{n(k+1)}\big(\ov{L}_f^{k+1}\cdot V\big),\]
and
\[\deg_{k,L}(f^n,V)=\big((f^n)^*L^k\cdot V\big)=q^{nk}\big(L^k\cdot V\big).\]

If $h_{\ov{L}_f}(V)=0$, then $\big(\ov{L}_f^{k+1}\cdot V\big)=0$ by the definition of height.
Consequently, by Lemma \ref{nefcompute}, the arithmetic degree of $(f,V)$ can be expressed as
\[\alpha_{k+1}(f,V)=\lim_{n\ra \infty}\max\left\{\widehat{\deg}_{k+1,\ov{L}_f}(f^n,V)^{1/n},\deg_{k,L}(f^n,V)^{1/n}\right\}=\lim_{n\ra \infty}q^k(L^k\cdot V)^{1/n}=q^k.\]

If $h_{\ov{L}_f}(V)>0$, then $\big(\ov{L}_f^{k+1}\cdot V\big)>0$. This implies that
\[\alpha_{k+1}(f,V)=\lim_{n\ra \infty}\max\left\{\widehat{\deg}_{k+1,\ov{L}_f}(f^n,V)^{1/n},\deg_{k,L}(f^n,V)^{1/n}\right\}=\lim_{n\ra \infty}q^{k+1}(\ov{L}_f^{k+1}\cdot V)^{1/n}=q^{k+1}.\]
\end{proof}

In \cite[Theorem 2.4]{zhang_small_1995}, Zhang proved that a point $P\in X$ is preperiodic if and only if $h_{\ov{L}_f}(P)=0$. He also proposed a conjecture that states that a subvariety $V$ with $\dim V\geq 1$ is preperiodic if and only if $h_{\ov{L}_f}(V)=0$. However, in \cite{ghioca_towards_2011}, a counterexample was constructed using elliptic curves with complex multiplication. We will demonstrate that it also serves as a counterexample to Conjecture \ref{Kawa-Sil}(2).

Let $E: y^2=x^3+x$ be an elliptic curve with complex multiplication over $K=\mb{Q}(\mr{i})$. There is an isomorphism $\iota: \mb{Z}[\mr{i}]\ra \mr{End}(E)$ given by $\iota(\omega)=[\omega]$. The endomorphism $[\mr{i}]$ is defined by $(x,y)\mapsto (-x,\mr{i}y)$. Let $f=([2-\mr{i}],[2+\mr{i}])$ be an endomorphism of $E\times_K E$ and $\Delta=\{(x,x)|x\in E\}$ be the diagonal of $E\times_K E$. Then we have the following corollary.
\begin{cor}
With all the notations as above, then the $f$-orbit of $\Delta$ is Zariski dense in $E\times_K E$ and $\alpha_2(f,\Delta)<\alpha_2(f)$. 
\end{cor}
\begin{proof}
Firstly, by properties of isogenies, we can compute the degree $\deg([2+\mr{i}])=(\mb{Z}[\mr{i}]:(2+\mr{i}))=5$ and $\deg([2-\mr{i}])=5$. 
	
Consider an ample line bundle $M$ satisfying $[2+\mr{i}]^*M=5M$ and $[2-\mr{i}]^*M=5M$. Let $p_1$ and $p_2$ be the projections from $E\times_K E$ to the first and second coordinates, respectively, and set $L=p_1^*M+p_2^*M$. Then $L$ is an ample line bundle and the map $f=([2+\mr{i}],[2-\mr{i}])$ is a polarized endomorphism with $f^*L=5L$. Since $\frac{2+\mr{i}}{2-\mr{i}}$ is not a root of unity, we know from \cite[Theorem 1.2]{ghioca_towards_2011} that the preperiodic points are dense in $\Delta$, and $\Delta$ itself is not preperiodic. Then the orbit of $\Delta$ is Zariski dense in $E\times_K E$ since $\mr{codim}\,\Delta=1$.

We now proceed to compute the arithmetic degree of the endomorphism $f$. According to  \cite[Theorem 1.10]{zhang_small_1995}, if $h_{\overline{L}_f}(\Delta)\neq 0$, there must exist a Zariski open set $U$ of $\Delta$ such that $h_{\overline{L}_f}$ on $U(\ov{K})$ has a positive lower bound, and any preperiodic point of $\Delta$ will be contained in $\Delta-U$. However, since prerperiodic points are dense in $\Delta$, there is a contradiction. Hence $h_{\overline{L}_f}(\Delta)=0$. 
		
By Proposition \ref{egpolarized}, we get $\alpha_{2}(f,\Delta)=5<25=\alpha_{2}(f)$.
\end{proof}

Since Conjecture \ref{Kawa-Sil}(2) is false, it becomes pertinent to investigate the specific criteria that guarantee the equality $\alpha_{k+1}(f,V)=\alpha_{k+1}(f)$ between these two arithmetic degrees. In the case of polarized endomorphisms, this equality holds if the canonical height of $V$ is positive, which is closely connected to the dynamical Manin-Mumford conjecture. However, for general rational self-map $f$, the question remains open, prompting us to propose the following problem:
\begin{prb}
Consider a normal projective variety $X$ over $\ov{\mb{Q}}$, and let $f: X\dashrightarrow X$ be a dominant rational self-map. For $k>0$, what geometric or dynamical conditions on $k$-dimensional irreducible subvarieties $V\subset X$ lead to the equality $\alpha_{k+1}(f,V)=\alpha_{k+1}(f)$?
\end{prb}

	\bibliography{REFv3}
	\bibliographystyle{unsrt}
\end{document}